\documentclass{amsart}
\usepackage{amssymb,amscd}
\usepackage{tikz}
\usetikzlibrary{calc}
\usepackage[utf8]{inputenc}

\normalsize

\theoremstyle{plain}
\newtheorem*{thma}{Theorem A}
\newtheorem*{thmb}{Theorem B}
\newtheorem*{thmc}{Theorem C}
\newtheorem*{thmd}{Theorem D}
\newtheorem*{thme}{Theorem E}
\newtheorem*{thmf}{Theorem F}
\newtheorem*{lemg}{Lemma G}
\newtheorem*{thmh}{Theorem H}
\newtheorem{theorem}{Theorem}[section]
\newtheorem{lemma}[theorem]{Lemma}

\newtheorem{corollary}[theorem]{Corollary}

\theoremstyle{remark}

\newtheorem{remark}[theorem]{Remark}
\newtheorem{remark*}{Remark}

\newcommand{\Z}{\mathbb{Z}}
\newcommand{\R}{\mathbb{R}}
\newcommand{\C}{\mathbb{C}\mkern1mu}
\renewcommand{\H}{\mathbb{H}\mkern1mu}

\newcommand{\RP}{\mathbb{R\mkern1mu P}}
\newcommand{\CP}{\mathbb{C\mkern1mu P}}
\newcommand{\HP}{\mathbb{H\mkern1mu P}}
\newcommand{\OP}{\mathbb{O\mkern1mu P}}
\newcommand{\Sph}{\mathbb{S}}

\DeclareMathOperator{\tr}{trace}
\DeclareMathOperator{\id}{id}

\DeclareMathOperator{\Ad}{Ad}
\DeclareMathOperator{\ad}{ad}

\DeclareMathOperator{\codim}{codim}

\newcommand{\tp}{^{\mathrm{t}}}

\newcommand{\one}{{\mathchoice
{\mathrm 1\mskip-4.2mu\mathrm l}{\mathrm 1\mskip-4.2mu\mathrm l}
{\mathrm 1\mskip-3.9mu\mathrm l}{\mathrm 1\mskip-4.0mu\mathrm l}}}
\newcommand{\bmat}{\left(\begin{smallmatrix}}
\newcommand{\emat}{\end{smallmatrix}\right)}
\newcommand{\abs}[1]{\vert #1\vert}

\newcommand{\SO}{\mathrm{SO}}
\newcommand{\OO}{\mathrm{O}}
\newcommand{\SU}{\mathrm{SU}}

\newcommand{\Syp}{\mathrm{Sp}}
\newcommand{\Spin}{\mathrm{Spin}}
\newcommand{\Gtwo}{\mathrm{G}_2}
\newcommand{\Ffour}{\mathrm{F}_4}

\newcommand{\ab}{\mathfrak{a}}
\newcommand{\af}{\mathfrak{f}}
\newcommand{\ag}{\mathfrak{g}}
\newcommand{\ah}{\mathfrak{h}}
\newcommand{\ak}{\mathfrak{k}}
\newcommand{\am}{\mathfrak{m}}

\newcommand{\ap}{\mathfrak{p}}
\newcommand{\aq}{\mathfrak{q}}
\newcommand{\aso}{\mathfrak{so}}

\DeclareMathOperator{\spann}{span}
\newcommand{\nor}{\mathrm{nor}}

\begin{document}

\title{Harmonic self-maps of cohomogeneity one manifolds}
\author{Thomas P\"uttmann}
\address{Ruhr-Universit\"at Bochum\\
Fakult\"at f\"ur Mathematik\\
44780 Bochum\\
Germany}
\email{Thomas.Puettmann@rub.de}
\author{Anna Siffert}
\thanks{The second named author was supported by DFG-grant SI 2077/1-1. She would also like to thank the Max-Planck-Institut Bonn for financial support and excellent working conditions. Both authors would like to thank Wolfgang Ziller for valuable comments.}
\address{Max-Planck-Institut f\"ur Mathematik\\
Vivatsgasse 7\\
53111 Bonn\\
Germany}
\email{Siffert@mpim-bonn.mpg.de}


\begin{abstract}
We develop the theory of equivariant harmonic self-maps of compact cohomogeneity one manifolds and construct new harmonic self-maps of the compact Lie groups $\mathrm{SO}(4\ell+2)$, $\ell \ge 1$ with degree~$-3$, of $\mathrm{SO}(8)$, $\SO(14)$ and $\mathrm{SO}(26)$ with degree $-5$ each, of $\mathrm{SO}(10)$ with degree $-7$, and of $\mathrm{SO}(14)$ with degree $-11$ by exhibiting linear solutions to non-linear singular boundary value problems.
\end{abstract}

\maketitle

\section{Introduction}
In this paper we develop the theory of equivariant harmonic self-maps of compact cohomogeneity one manifolds and construct new harmonic self-maps of the compact Lie groups $\SO(4\ell+2)$, $\ell \ge 1$ with degree~$-3$, of $\SO(8)$, $\SO(14)$ and $\SO(26)$ with degree $-5$ each, of $\SO(10)$ with degree $-7$, and of $\SO(14)$ with degree $-11$ by exhibiting linear solutions to non-linear singular boundary value problems.

\smallskip

Topologically non-trivial self-maps of compact manifolds are difficult to construct. It is natural to look for constructions in the presence of symmetries. Since every equivariant self-map of a compact homogeneous space is a diffeomorphism the homogeneous setting is too restrictive to be of interest. The next step is to consider cohomogeneity one manifolds.

\smallskip

Urakawa \cite{urakawa} calculated the tension fields of equivariant maps between cohomogeneity one manifolds and constructed some new examples of harmonic maps. For our purposes, however, the hypothesis on the actions and the invariant metrics given in this paper are too restrictive. By our more geometric approach, we extend the results of \cite{urakawa} in the context of self-maps in several directions. Most notably, we employ the construction of topologically non-trivial self-maps of cohomogeneity one manifolds given by the first named author in \cite{puttmann}.

\smallskip

Let $M$ be a Riemannian manifold with an isometric action $G\times M\to M$ of a compact Lie group $G$ such that the orbit space $M/G$ is isometric to a closed interval $[0,L]$ and such that the Weyl group $W$ of the action is finite.
In \cite{puttmann} the first named author constructed an infinite family of equivariant self-maps of $M$ by mapping $g\cdot\gamma(t)$ to $g\cdot\gamma(kt)$. 
Here, $\gamma$ is a unit speed normal geodesic such that $\gamma(0)$ is contained in one of the non-principal orbits. 
The integer $k$ is of the form $j\abs{W}/2+1$ with $j\in 2\Z$ (depending on the action odd integers $j$ might also be allowed). We call the map $g\cdot\gamma(t) \mapsto g\cdot\gamma(kt)$ the {\em $k$-map} of $M$.
The  degree of a $k$-map is equal to $k$ if the codimensions of the non-principal orbits are both odd, and equal to $0$ or $\pm 1$ otherwise.

Given $M$, there is the natural question of whether some of the $k$-maps are harmonic. More generally, we consider the equivariant {\em $(k,r)$-maps}, i.e., the maps
\begin{gather*}
  g\cdot\gamma(t)\mapsto g\cdot\gamma(r(t))
\end{gather*}
where $r: [0,L] \to \R$ is a smooth function with $r(0) = 0$ and $r(L) = kL$. 
 Any $(k,r)$-map is clearly equivariantly homotopic to the corresponding $k$-map. A $(k,r)$-map of $M$ is harmonic if and only if its tension field $\tau$ vanishes. The tension field splits into two natural components, the component $\tau^{\tan}$ tangential to the orbits and the component $\tau^{\nor}$ perpendicular to the orbits.
In order to state the normal component of the tension field in a computationally convenient way, we need to introduce some notation. Let
\begin{gather*}
  \Pi_t^{r(t)}: T_{\gamma(t)} (G\cdot\gamma(t)) \to T_{\gamma(r(t))} (G\cdot \gamma(r(t))
\end{gather*}
denote the parallel transport along the normal geodesic~$\gamma$. We have another natural but not neccessarily isometric homomorphism between the two tangent spaces $T_{\gamma(t)} (G\cdot\gamma(t))$ and $T_{\gamma(r(t))} (G\cdot\gamma(r(t)))$, namely, the action field homomorphism given by $X^{\ast}_{\vert\gamma(t)} \mapsto X^{\ast}_{\vert \gamma(r(t))}$. Let $J_t^{r(t)}$ denote the endomorphism of $T_{\gamma(t)} (G\cdot\gamma(t))$ given by composing the action field homomorphism with $(\Pi_t^{r(t)})^{-1} = \Pi_{r(t)}^t$.

\begin{thma}
The normal component of the tension field of a $(k,r)$-map of $M$ is given by
\begin{gather*}
  \tau^{\nor}_{\vert\gamma(t)} = \ddot r(t) - \dot r(t) \tr S_{\vert \gamma(t)}
    + \tr\, (J_t^{r(t)})^{\ast} (\Pi_t^{r(t)})^{-1} S_{\vert \gamma(r(t))} \Pi_t^{r(t)} J_t^{r(t)}
\end{gather*}
for $0 < t < L$. Here, $S_{\vert \gamma(t)}$ denotes the shape operator of the orbit $G\cdot \gamma(t)$ at $\gamma(t)$ and $(J_t^{r(t)})^{\ast}$ denotes the adjoint endomorphism of $J_t^{r(t)}$.
\end{thma}

\smallskip

The principal isotropy groups $H = G_{\gamma(t)}$ along the normal geodesics are constant for $0 < t < L$. Let $Q$ be a fixed biinvariant metric on $G$ and let $\mathfrak{n}$ denote the orthonormal complement of the Lie algebra $\ah$ of the principal isotropy group $H$ in~$\ag$.

\begin{thmb}
The tangential component of the tension field of a $(k,r)$-map of $M$ is given by
\begin{gather*}
  \tau^{\tan}_{\vert\gamma(t)} = -\sum_{\mu,\nu=1}^n \langle [E_{\mu},F_{\nu}]^{\ast},E_{\mu}^{\ast} \rangle_{\vert\gamma(r(t))} F^{\ast}_{\nu\vert\gamma(r(t))}
\end{gather*}
for $0 < t < L$. Here, $E_1,\ldots,E_n \in \mathfrak{n}$ and $F_1,\ldots,F_n \in \mathfrak{n}$ are such that $E^{\ast}_{1\vert\gamma(t)}$,\ldots, $E^{\ast}_{n\vert\gamma(t)}$ form an orthonormal basis of $T_{\gamma(t)}(G\cdot\gamma(t))$ and $F^{\ast}_{1\vert\gamma(r(t))},\ldots, F^{\ast}_{n\vert\gamma(r(t))}$ form an orthonormal basis of $T_{\gamma(r(t))}(G\cdot\gamma(r(t)))$.
\end{thmb}

For self-maps, Theorem\,A and Theorem\,B are a more general and more geometric version of Theorem\,2.2 of \cite{urakawa}.

\smallskip

A $(k,r)$-map of $M$ is harmonic if both, the normal and the tangential component of the tension field vanish. Theorem\,A leads to the non-linear singular boundary value problem
\begin{gather*}
  0 = \ddot r(t) - \dot r(t) \tr S_{\vert \gamma(t)}
    + \tr\, (J_t^{r(t)})^{\ast} (\Pi_t^{r(t)})^{-1} S_{\vert \gamma(r(t))} \Pi_t^{r(t)} J_t^{r(t)}
\end{gather*}
for functions $r: \,]0,L[\, \to \R$ with $\lim_{t\to 0} r(t) = 0$ and $\lim_{t\to L} r(t) = kL$. If $r$ is a solution of this boundary value problem, the condition provided by Theorem\,B is an infinite set of algebraic equations. It is hence clear that the tangential component of the tension field of a $(k,r)$-map will only vanish in rather special geometric situations. 
Simple examples where the tangential components fail to vanish for all $(k,r)$-maps except the identity are provided by 
$\Sph^2 \times \Sph^2_{\rho}$ where $\Sph^2_{\rho}$ denotes the sphere of radius $\rho$ and $\rho^2$ is a rational number $\neq 1$.

\smallskip

The special geometric situations that we consider here are the cohomogeneity one actions on spheres and their lifts to the orthogonal groups that act transitively on the spheres. For most of these actions the tangential components of the tension fields turn out to vanish for all $(k,r)$-maps, no matter if $r$ solves the boundary value problem or not. Moreover, there are at most two eigenvalues of the Jacobi operator along a normal geodesic. This fact is of central importance for the practical evaluation of the ODE given in Theorem\,A.

\smallskip

The cohomogeneity one actions on spheres were classified by Hsiang and Lawson \cite{hsiang}. They showed that each of these actions is orbit equivalent to the isotropy representation of a Riemannian symmetric space of rank $2$.
The orbits of any isometric cohomogeneity one action $G\times \Sph^{n+1} \to \Sph^{n+1}$ yield an isoparametric foliation of the sphere. Takagi and Takahashi \cite{takagi} determined the number $g$ of distinct principal curvatures of the orbits and their multiplicities $m_0,\ldots,m_{g-1}$. It turned out that
\begin{gather*}
  m_0 = m_2 = \ldots = m_{g-2} \quad \text{and} \quad m_1 = m_3 = \ldots = m_{g-1}.
\end{gather*}
In particular, $n = \frac{m_0+m_1}{2} g$. M\"unzner \cite{munzner} later showed that this is a general property of isoparametric foliations of spheres. Up to ordering of $m_0$ and $m_1$ there are only actions with the following $(g,m_0,m_1)$:
\begin{multline*}
 (1,m,m), (2,m_0,m_1), (3,1,1), (3,2,2), (3,4,4), (3,8,8),\\
   (4,m_0,1), (4,2,2), (4,2,2\ell+1), (4,4,4\ell+3), (4,4,5), (4,6,9), (6,1,1), (6,2,2).
\end{multline*}
The classification shows that all cohomogeneity one actions with given data $(g,m_0,m_1)$ are orbit equivalent except in the case $(4,2,1)$ where two different classes exist. An explicit list with detailed information can be found in \cite{gwz}.

\smallskip

We call a cohomogeneity one action on a sphere with the data $(g,m_0,m_1)$ briefly a {\em $(g,m_0,m_1)$-action}. If $m_0 = m_1 =: m$ (by M\"unzner's results, this is necessarily true for odd $g$) we simply call the action a $(g,m)$-action. Note that for each of all these actions the Weyl group is the dihedral group $D_g$ of order $2g$ and there are $k$-maps for all integers $k = jg+1$ with $j\in \Z$.

\begin{thmc}
Given a $(g,m_0,m_1)$-action on $\Sph^{n+1}$ with $n = \frac{m_0+m_1}{2}g$, a $k$-map is harmonic if and only if $k=1$, or $(g=2$ and $k=-1)$, or $(m_0 = m_1$ and $k = 1-g)$.
\end{thmc}

With this theorem we reproduce some known harmonic self-maps of spheres by a different method: 
In \cite{siffertone} the $(1-g)$-map of $\Sph^{mg+1}$ is identified with the gradient map of the Cartan-M\"unzner polynomial of the associated isoparametric foliation. 
It was previously known that the Cartan-M\"unzner polynomial and its gradient map are harmonic if all multiplicities are equal \cite{eellsrato}.  The degrees of the gradient map were previously obtained in \cite{peng} and \cite{ge}.

\smallskip

Any isometric cohomogeneity one action $G\times \Sph^{n+1}\to \Sph^{n+1}$ can be lifted to an isometric cohomogeneity one action of $G\times\SO(n+1)$ on $\SO(n+2)$ with the metric $\tfrac{1}{2} \tr X\tp Y$ by
\begin{gather*}
  G \times \SO(n+1) \times \SO(n+2) \to \SO(n+2), \quad \left( A,\bmat 1 & \\ & B\emat \right) \cdot C = AC\bmat 1 & \\ & B^{-1}\emat.
\end{gather*}
We call the lift of a $(g,m_0,m_1)$-action on $\Sph^{n+1}$ a {\em $(g,m_0,m_1)$-action on $\SO(n+2)$} (or simply a $(g,m)$-action if $m_0 = m_1$). This definition fits to the fact that the Weyl group of the lifted action is again the dihedral group $D_g$. In contrast to the actions on spheres, $k$-maps exist for the lifted actions only for all integers $k = jg+1$ with $j\in 2\Z$.

\begin{thmd}
Given a $(g,m_0,m_1)$-action on $\SO(n+2)$ with $n = \frac{m_0+m_1}{2}g$, a $k$-map is harmonic if and only if $k = 1$ or $(m_0 = m_1$ and $k=1-2g)$.
\end{thmd}

This theorem is the source of the new concrete examples of harmonic self-maps mentioned in the opening of this introduction: the $-3$-maps of the $(2,m)$-actions on $\SO(2m+2)$ with degree $+1$ if $m$ is odd and degree $-3$ if $m$ is even; the $-5$-maps of the $(3,m)$-actions on $\SO(5)$, $\SO(8)$, $\SO(14)$, and $\SO(26)$ with degree $+1$ in the first case and $-5$ in the other cases; the $-7$-maps of the $(4,m)$-actions on $\SO(6)$ with degree $+1$ and on $\SO(10)$ with degree $-7$; the $-11$-maps of the $(6,m)$-actions on $\SO(8)$ with degree $+1$ and on $\SO(14)$ with degree $-11$.

\smallskip

In order to prove Theorems\,C and~D we need to evaluate the general expressions for the normal and tangential components of the tension fields of the $(k,r)$-maps given in Theorem\,A and~B for the $(g,m_0,m_1)$-actions on $\Sph^{n+1}$ and $\SO(n+2)$. For the normal component, a non-standard trigonometric identity (Lemma\,\ref{2parac}) is the key to obtain the following result.

\begin{thme}
Given a $(g,m_0,m_1)$-action on a sphere $\Sph^{n+1}$ with $n = \frac{m_0+m_1}{2}g$, the normal component of the tension field of a $(k,r)$-map vanishes if and only if $r$ satisfies the {\em $(g,m_0,m_1,k)$-boundary value problem}
\begin{multline*}
  0 =
  4\sin^{2} gt \cdot \ddot r(t) + \bigl( g(m_0+m_1)\sin 2gt + 2g(m_0-m_1)\sin gt \bigr)\dot r(t)\\
  - g(g -2)\sin 2(r(t)-t ) \bigl( m_0+m_1 + (m_0-m_1)\cos gt \bigr)\\
  -2g \sin\bigl(2(r(t)-t )+gt\bigr) \bigl( (m_0+m_1)\cos gt +m_0-m_1\bigr)
\end{multline*}
for functions $r: \,]0,\frac{\pi}{g}[\, \to \R$ with
\begin{gather*}
  \lim_{t\rightarrow 0}r(t)=0 \quad \text{and}\quad \lim_{t\rightarrow \tfrac{\pi}{g}}r(t)=k\tfrac{\pi}{g}
\end{gather*}
\end{thme}
Note that in the special case where $m_0 = m_1 = m$ the ODE above simplifies to
\begin{multline}
\label{ode}
0 = 2\sin^2 gt \cdot\ddot r(t)
  + mg\sin 2gt \cdot\dot r(t) \\ -mg\big((g-1)\sin 2(r(t)-t) + \sin 2(r(t)+(g-1)t)\big).
\end{multline}

\smallskip

It is remarkable that the same boundary value problems also appear for the lifted actions. Due to the number of non-zero distinct principal curvatures and their multiplicities, it is not the $(g,m_0,m_1)$-action on $\Sph^{n+1}$ that is related to the $(g,m_0,m_1)$-action on $\SO(n+2)$ this way but the $(2g,m_0,m_1)$-action on $\Sph^{2n+1}$ (if it exists). In order to make this connection visible we have to reparametrize the normal geodesics by a factor of $2$.

\begin{thmf}
Given a $(g,m_0,m_1)$-action on $\SO(n+2)$ with $n = \frac{m_0+m_1}{2}g$, the normal component of the tension field of the map $g\cdot \tilde\gamma(2t)\to g\cdot\tilde\gamma(2r(t))$ vanishes if and only if $r$ solves the $(2g,m_0,m_1,k)$-boundary value problem.
\end{thmf}

\medskip

In this paper we just look for linear solutions of the $(g,m_0,m_1,k)$-boundary value problems. 

\begin{lemg}
The linear function $r(t)=kt$ is a solution of the $(g,m_0,m_1,k)$-boundary value problem if and only if $k=1$, or $(g=2$ and $k=-1)$, or $(m_0=m_1$ and $k=1-g)$.
\end{lemg}

There is the natural question of whether some of the $(g,m_0,m_1,k)$-boundary value problems have nonlinear solutions. The cases $g=1$, $m_0 = m_1 = m$, $k=\pm 1$ was treated in detail by Bizon and Chmaj \cite{bizon} in a different language. For each $m\in\{2,3,4,5\}$ and each $k \in \{0,1\}$ they constructed an countably infinite family of equivariant harmonic self-maps of $\Sph^{m+1}$ with degree $k$ and showed that such a family could not exist for
$m \ge 6$. In \cite{sifferttwo} the second named author of the current paper constructed infinite families of nonlinear solutions of the $(2,m_0,m_1,\pm 1)$-boundary value problems. In a further subsequent paper the second named author will in particular construct nonlinear solutions of the $(3,2,2,-2)$-, $(4,2,2,5)$-, $(6,2,2,7)$-, $(g,1,1,1+g)$- and the $(g,1,1,1-2g)$-boundary value problems for $g\in \{2,3,4,6\}$. These solutions yield further new harmonic self-maps of the corresponding spheres and orthogonal groups.

\medskip

Finally, the general expression of the tangential part of the tension fields is evaluated by two different strategies. Using Schur's Lemma multiple times, we first show that the tangential part vanishes for actions for which the isotropy representations of the principal orbits $G/H$ decompose into inequivalent irreducible $H$-modules. This works for $g \le 3$ and for $(g,m) = (4,2)$ and $(g,m) = (6,2)$. The second, more general strategy is to reduce the computations by determining the fixed point sets of the principal isotropy groups and by employing the action of the Weyl group. We work this strategy out for $(g,m_0,m_1) = (4,m_0,1)$ and $(g,m) = (6,1)$.

\begin{thmh}
Given a $(g,m_0,m_1)$-action on $\Sph^{n+1}$ or $\SO(n+2)$ the tangential component of the tension field of any $(k,r)$-map vanishes except possibly for
\begin{gather*}
  (g,m_0,m_1) \in \{ (4,2,2\ell+1), (4,4,4\ell+3), (4,4,5), (4,6,9) \}.
\end{gather*}
\end{thmh}
In the remaining cases the function $r(t) = kt$ is not a solution of the $(g,m_0,m_1,k)$-boundary value problem and hence the normal component of the tension field of the $k$-map does not vanish. This does not mean, however, that there are no harmonic $(k,r)$-maps. Since the focus of the current paper is on linear solutions, we have avoided the lengthy computation needed to answer the question whether the tangential components in the remaining cases vanish for all $(k,r)$-maps.

\smallskip

We finally note that, somewhat exceptionally, the $(6,1)$-action on $\Sph^7$ lifts to a cohomogeneity one action on the compact Lie group $\Syp(2)$. The Weyl group of the lifted action is again the dihedral group $D_6$. In this case, $k$-maps exist for all integers $k = 6j+1$ with $j\in \Z$. The tangential component of the tenison field of any $(k,r)$-map on $\Syp(2)$ vanishes since the fixed point set of the principal isotropy group just consists of one unparametrized normal geodesic. The normal component vanishes if $r$ solves the $(6,1,1,k)$-boundary value problem. Hence, a $k$-map of $\Syp(2)$ is harmonic if and only if $k=1$ or $k= -5$ (the degree of the $-5$-map is $+1$).

\smallskip

The paper is organized as follows: In Section\,\ref{kmaps} we review the construction of the equivariant self-maps of cohomogeneity one manifolds introduced by the first named author in~\cite{puttmann}. In Section\,\ref{tension} we compute the tensions fields of the $(k,r)$-maps in the general setting and prove Theorem\,A and Theorem\,B. In Section\,\ref{spheres} and Section\,\ref{groups} we evaluate the general expression for the normal component of the tension field in the special case of the $(g,m_0,m_1)$-actions on $\Sph^{n+1}$ and $\SO(n+2)$, respectively. Theorem\,C is proved in Section\,\ref{spheres}, Theorem\,C is proved in Section\,\ref{groups}.  In Section\,\ref{linearsolutions} we investigate by elementary means when the linear function $r(t) = kt$ solves the $(g,m_0,m_1,k)$-boundary value problem and prove Lemma\,G. In Section\,\ref{tangentialparts} we evaluate the tangential components of the tension fields for the $(g,m_0,m_1)$-actions on $\Sph^{n+1}$ and $\SO(n+2)$ and prove Theorem\,H.
The nonstandard trigonometric identity used in the proofs of Theorem\,E and Theorem\,F is established in Section\,\ref{trigidentity}.

\bigskip

\section{Equivariant self-maps of cohomogeneity one manifolds}
\label{kmaps}
In this initial section we briefly review the construction of the equivariant self-maps of cohomogeneity one manifolds introduced by the first named author in~\cite{puttmann}.

\smallskip

Let $M$ be as in the introduction, i.e., a compact Riemannian manifold with an isometric action $G\times M\to M$ of a compact Lie group $G$ on $M$ such that the orbit space $M/G$ is isometric to the closed interval $[0,L]$. The end points $0$ and $L$ of the interval correspond to non-principal orbits $N_0$ and $N_1$ while each interior point $t$ corresponds to a principal orbit. We fix a unit-speed normal geodesic $\gamma$, i.e., a geodesic $\gamma:\R\to M$ with $\gamma(0)\in N_0$ and $\gamma(L) \in N_1$ that passes through all orbits perpendicularly. The isotropy groups of the regular points $\gamma(t)$ with $t\not\in \Z L$ are constant. We denote this common principal isotropy group by $H$. The Weyl group $W$ is by definition the subgroup of the elements of $G$ that leave $\gamma$ invariant modulo the subgroup of elements that fix $\gamma$ pointwise. The Weyl group $W$
is a dihedral subgroup of $N(H)/H$ generated by two involutions that fix $\gamma(0)$ and $\gamma(L)$, respectively. It acts simply transitively on the regular segments of the normal geodesic. We assume that $\gamma$ is closed or, equivalently, that $W$ is finite.

\begin{theorem}[see \cite{puttmann}]
The assignment $g\cdot \gamma(t) \mapsto g\cdot \gamma(kt)$ leads to a well defined smooth self-map of $M$, the {\em $k$-map}, if $k$ is of the form $k = j\abs{W}/2+1$ where $j$ is any even integer. This is even true for any integer $j$ if the isotropy group at $\gamma(L)$ is a subgroup of the isotropy group at $\gamma\bigl((\abs{W}/2+1)L\bigr)$. The degree of the $k$-map is given by
\begin{gather*}
  \deg = \begin{cases}
  k & \text{if $\codim N_0$ and $\codim N_1$ are both odd,}\\
  +1 & \text{otherwise,}
  \end{cases}
\end{gather*}
if $j$ is even, and by
\begin{gather*}
  \deg = \begin{cases}
  k & \text{if $\codim N_0$ and $\codim N_1$ are both odd,}\\
  0 & \text{if $\codim N_0$ and $\codim N_1$ are both even, $\abs{W}\not\in 4\Z$,}\\
  -1 & \text{if $\codim N_0$ is even, $\codim N_1$ is odd, and $\abs{W}\not\in 8\Z$,}\\
  +1 & \text{otherwise,}
  \end{cases}
\end{gather*}
if $j$ is odd.
\end{theorem}

\begin{figure}[b]
\begin{tikzpicture}
\node at (0,0) {$\SU(3),\SO(2)$};
\draw (0,0) circle [radius=1.6cm]; 
\draw(0:1.5cm)--(0:1.7cm);
\node at (0:2.2cm) {$\SO(3)$};
\draw(90:1.5cm)--(90:1.7cm);
\node at (90:2.0cm) {$\SU(2)$};
\draw(180:1.5cm)--(180:1.7cm);
\node at (180:3.6cm) {$\bmat i & &\\ & i & \\ & & -1\emat\SO(3)\bmat -i & & \\ & -i & \\ & & 1\emat$};
\draw(270:1.5cm)--(270:1.7cm);
\node at (270:2.0cm) {$\SU(2)$};
\end{tikzpicture}
\caption{Extended group diagram for the action  $\SU(3)\times\SU(3) \to \SU(3)$, $(A,B) \mapsto ABA\tp$.}
\label{extended}
\end{figure}
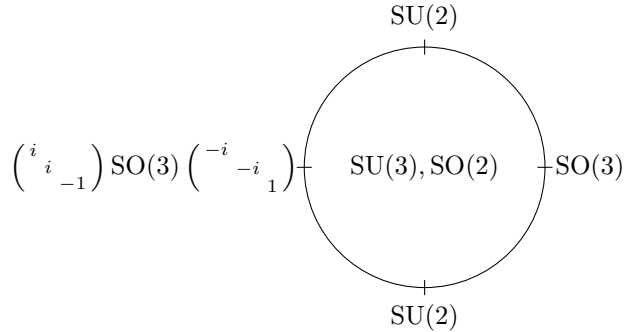

When working with this construction it is very convenient to use extended group diagrams. The usual group diagram of the action $G\times M\to M$ consists of the group $G$, the principal isotropy group $H$ along a fixed normal geodesic $\gamma$ and the two non-principal isotropy groups $K_0 = G_{\gamma(0)}$ and $K_1 = G_{\gamma(L)}$. In the {\em extended group diagram} we draw a circle for $\gamma$, denote the pair $(G,H)$ in the center of the circle and denote the non-principal isotropy groups $K_{i}$ at the positions $\gamma(iL)$ where they occur. We choose $\gamma$ to start on the right and to proceed counter-clockwise.

As an example we consider the extended group diagram of the action $\SU(3)\times\SU(3) \to \SU(3)$,
$(A,B) \mapsto ABA\tp$ with the unit-speed normal geodesic
\begin{gather*}
  \gamma(t) = \bmat \cos t & -\sin t & 0\\ \sin t & \cos t & 0\\ 0 & 0 & 1\emat
\end{gather*}
and $L=\pi/2$. The isotropy groups $K_0$ and $K_1$ are given by $\SO(3)$ and $\SU(2)$, respectively. The Weyl group $W$ is a dihedral group of order $\abs{W} = 4$ generated by the two involutions $\sigma_0 = \bmat 1 & & \\ & -1 & \\ & & -1\emat$ and $\sigma_1 = \bmat i & & \\ & -i & \\ & & 1\emat$. The isotropy groups $K_2$ and $K_3$ at $t=\pi$ and $t= 3\pi/2$ are equal to $\rho K_0\rho^{-1}\neq K_0$ and $\rho K_1\rho^{-1} = K_1$, respectively, where $\rho = \sigma_1\circ\sigma_0$, see Figure\,\ref{extended}. From this diagram we see that we get $k$-maps of $\SU(3)$ for all odd integers $k$. Note that if we use the normal geodesic $\gamma(t-\pi/2)$, i.e., if we start at the singular orbit $N_1$ instead of at $N_0$ we only get $k$-maps for $k \in 4\Z-1$. Each of these $k$-maps is identical to the $k$-map constructed starting from $N_0$. In general, it matters only for odd integers $j$ at which non-principal orbit the normal geodesics start.

\bigskip

\section{The tension field of the $(k,r)$-maps}
\label{tension}
In this section we compute the tension field $\tau$ of the $(k,r)$-maps defined in the introduction. 
The normal component is given in terms of the shape operator of the orbits while the formula for the tangential component involves more information on the group action.

\smallskip

Below we let $\psi$ be a given $(k,r)$-map, i.e., $\psi(g\cdot \gamma(t))=g\cdot \gamma(r(t))$.
In order to compute the tension field we need the normal and tangential derivatives of $\psi$. The derivative of $\psi$ in normal direction is given by
\begin{gather*}
  d\psi_{\vert \gamma(t)} \cdot \dot \gamma(t) = \frac{d}{dt}(\psi\circ\gamma(t)) = \dot r(t) \dot \gamma(r(t)).
\end{gather*}
In order to compute the derivative of $\psi$ in tangential directions we use action fields. To each element $X$ of the Lie algebra $\ag$ of $G$ there is the corresponding action field $X^{\ast}$ on $M$ given by
\begin{gather*}
  X^{\ast}_{\vert p} = \frac{d}{ds} (\exp sX \cdot p)_{\vert s = 0}.
\end{gather*}
The map $\ag/\ah \to T_p(G\cdot p)$ is a vector space isomorphism for regular points $p$. Now we get
\begin{gather*}
  d\psi_{\vert\gamma(t)}\cdot X^{\ast}_{\vert \gamma(t)} = \frac{d}{ds} \psi(\exp sX \cdot \gamma(t))_{\vert s= 0}
  = \frac{d}{ds}\bigl(\exp sX \cdot \gamma(r(t))\bigr)_{\vert s=0} = X^{\ast}_{\vert\gamma(r(t))}.
\end{gather*}

\smallskip

By its very definition the tension field $\tau$ of $\psi$ is given by
\begin{gather}
   \tau_{\vert p} = \sum_{\mu=0}^n \nabla d\psi_{\vert p}(e_{\mu},e_{\mu}) = \sum_{\mu=0}^{n}\Bigl( \nabla_{e_{\mu}} (d\psi\cdot e_{\mu}) - d\psi \cdot \nabla_{e_{\mu}}e_{\mu}\Bigr)_{\vert p}
\label{tensiondef}
\end{gather}
where the vectors $e_0,\ldots,e_n$ form any orthonormal basis of $T_p M$ and can be extended arbitrarily to vector fields on a neighborhood of $p$. Because of the equivariance of the tension field it suffices to evaluate the expression along $\gamma(t)$. We denote by $T$ the unit normal field to the principal orbits given by $T_{\vert g\cdot \gamma(t)} = g\cdot \dot \gamma(t)$. Furthermore, we set $e_0 = \dot\gamma(t)$ and choose $E_1,\ldots, E_n\in \ag$ such that $e_1 = E^{\ast}_{1\vert\gamma(t)},\ldots,e_n = E^{\ast}_{n\vert\gamma(t)}$ form an orthonormal basis of $T_{\gamma(t)}(G\cdot\gamma(t))$. Now,
\begin{gather*}
  \nabla_{e_0}(d\psi\cdot e_0)_{\vert \gamma(t)} = \frac{\nabla}{dt} \bigl(\dot r(t) \dot \gamma(r(t))\bigr) = \ddot r(t) \dot \gamma(r(t))
\end{gather*}
and $(d\psi\cdot \nabla_{e_0} e_0)_{\vert \gamma(t)} = d\psi_{\vert \gamma(t)}\cdot \frac{\nabla}{dt}\dot\gamma(t) = 0$. Hence, the tension field definition (\ref{tensiondef}) becomes
\begin{gather}
   \tau_{\vert \gamma(t)} = \ddot r(t) \dot \gamma(r(t)) + \sum_{\mu=1}^{n}\Bigl( \nabla_{E^{\ast}_{\mu}} E^{\ast}_{\mu\vert\gamma(r(t))}  - d\psi_{\vert \gamma(t)} \cdot \nabla_{E^{\ast}_{\mu}}E^{\ast}_{\mu\vert \gamma(t)}\Bigr).
\label{tensionev}
\end{gather}

The following theorem gives the normal component of the tension field in a way that is suitable for actions on spaces where one has some information on the Jacobi fields (e.g. symmetric spaces). Recall that
\begin{gather*}
  \Pi_t^{r(t)}: T_{\gamma(t)} (G\cdot\gamma(t)) \to T_{\gamma(r(t))} (G\cdot \gamma(r(t))
\end{gather*}
denotes the parallel transport along the normal geodesic~$\gamma$ and that $J_t^{r(t)}$ denotes the endomorphism of
$T_{\gamma(t)} (G\cdot\gamma(t))$ given by composing the action field homomorphism $X^{\ast}_{\vert\gamma(t)} \mapsto X^{\ast}_{\vert \gamma(r(t))}$ with $(\Pi_t^{r(t)})^{-1} = \Pi_{r(t)}^t$. We are now able to prove Theorem\,A from the introduction.

\begin{theorem}
The normal component $\tau^{\nor}_{\vert\gamma(t)} = \langle\tau_{\vert\gamma(t)}, \dot \gamma (r(t))\rangle$ of the tension field is given by
\begin{gather*}
  \tau^{\nor}_{\vert\gamma(t)} = \ddot r(t) - \dot r(t) \tr S_{\vert \gamma(t)}
    + \tr\, (J_t^{r(t)})^{\ast} (\Pi_t^{r(t)})^{-1} S_{\vert \gamma(r(t))} \Pi_t^{r(t)} J_t^{r(t)}
\end{gather*}
where $S_{\vert \gamma(t)}$ denotes the shape operator of the orbit $G\cdot \gamma(t)$ at $\gamma(t)$ and $(J_t^{r(t)})^{\ast}$ denotes the adjoint endomorphism of $J_t^{r(t)}$.
\label{normalone}
\end{theorem}
\begin{proof}
We evaluate the tension field formula (\ref{tensionev}). First notice that
\begin{gather*}
  \langle \nabla_{X^{\ast}} Y^{\ast}, T \rangle =
    \langle X^{\ast}, S \cdot Y^{\ast}\rangle
\end{gather*}
by the definition of the shape operator. Hence,
\begin{gather*}
  \langle \sum_{\mu=1}^n \nabla_{E^{\ast}_{\mu}} E^{\ast}_{\mu}, T \rangle_{\vert\gamma(t)} =
  \sum_{\mu=1}^n \langle e_{\mu}, S_{\vert\gamma(t)} e_{\mu}\rangle = \tr S_{\vert\gamma(t)}.
\end{gather*}
Similarly, we get
\begin{multline*}
  \langle \sum_{\mu=1}^n \nabla_{E^{\ast}_{\mu}} E^{\ast}_{\mu}, T \rangle_{\vert\gamma(r(t))} =
  \sum_{\mu=1}^n \langle E^{\ast}_{\mu\vert\gamma(r(t))}, S_{\vert \gamma(r(t))} E^{\ast}_{\mu\vert\gamma(r(t))} \rangle\\
  = \sum_{i=1}^n \langle J_t^{r(t)} e_{\mu}, (\Pi_t^{r(t)})^{-1} S_{\vert\gamma(r(t))} \Pi_t^{r(t)} J_t^{r(t)} e_{\mu}\rangle.
\qedhere
\end{multline*}
\end{proof}

\begin{remark}
For actions on concretely given spaces, the shape operator can be computed by
$S\cdot X^{\ast}_{\vert\gamma(t)} = -\nabla_{\dot\gamma(t)} X^{\ast}$. This standard fact follows from applying that $\nabla$ is torsion free to the second derivatives of the map $(s,t) \mapsto \exp sX \cdot \gamma(t)$.
\end{remark}

\begin{remark}
The $(k,r)$-maps can be defined analogously for singular codimension one metric foliations with closed normal geodesics. Theorem\,\ref{normalone} is still valid in this situation. One only has to substitute the action fields by variation vector fields of variations by normal geodesics.
\end{remark}

We now give a second formula in terms of data on the acting group $G$. We adopt the notation from \cite{puettmann} and \cite{gz}.

\smallskip

Let $Q$ be a fixed biinvariant metric on $G$. Denote the orthonormal complement of the Lie algebra $\ah$ of the principal isotropy group $H$ in $\ag$ by $\mathfrak{n}$. Define the metric endomorphisms $P_t : \mathfrak{n} \to \mathfrak{n}$ by
\begin{gather*}
  Q(X, P_t\cdot Y) = \langle X^{\ast},Y^{\ast} \rangle_{\vert\gamma(t)}.
\end{gather*}
Note that each $P_t$ is symmetric with respect to $Q$ and $\Ad_H$ equivariant. We have
\begin{multline*}
  \langle X^{\ast},S\cdot X^{\ast}\rangle_{\vert\gamma(r(t))}
  = -\langle X^{\ast}, \nabla_T X^{\ast}\rangle_{\vert\gamma(r(t))}
  = -\tfrac{1}{2\dot r(t)} \tfrac{d}{dt} \langle X^{\ast},X^{\ast}\rangle_{\vert\gamma(r(t))} \\
  = -\tfrac{1}{2} Q(X,(\dot P)_{r(t)} X)
  = -\tfrac{1}{2} \langle X^{\ast}, (P_t^{-1}(\dot P)_{r(t)} X)^{\ast} \rangle_{\vert\gamma(t)}
\end{multline*}
and hence the following statement holds.

\begin{theorem}
The normal component of the tension field is given by
\begin{gather*}
  \tau^{\nor}_{\vert\gamma(t)} = \ddot r(t) + \tfrac{1}{2}\dot r(t) \tr P_t^{-1}\dot P_t
    - \tfrac{1}{2} \tr P_t^{-1} (\dot P)_{r(t)}.
\end{gather*}
\label{normaltwo}
\end{theorem}
Note that the objects of the formula in Theorem\,\ref{normalone} depend only on the orbit geometry of the action. Usually, however, orbit equivalent actions will yield different endomorphisms fields $P_t$.
For the $(g,m)$-actions with $g \ge 3$ on $\SO(mg+2)$, for example, the metric endomorphisms
$P_t$ do not diagonalize simultaneously while the shape operators $S_{\vert\gamma(t)}$ do diagonalize simultaneously in a suitable parallel orthonormal basis along~$\gamma$. For these actions, the formula of Theorem\,\ref{normalone} is much easier to evaluate.

\medskip

We now turn to the tangential component of the tension field. First of all, we have the following standard fact.
\begin{lemma}
\label{leminvariant}
The tangential component $\tau^{\tan}_{\vert\gamma(t)}$ of the tension field is contained in the common fixed point set $\bigl(T_{\gamma(t)}(G\cdot \gamma(t))\bigr)^{H}$ of the principal isotropy group $H$ on $T_{\gamma(t)}(G\cdot\gamma(t))$.
\label{tangential}
\end{lemma}
\begin{proof}
This follows from the equivariance of the tension field $h\cdot \tau_{\vert p} = \tau_{\vert hp} = \tau_{\vert p}$. 
\end{proof}

We recall from formula (\ref{tensionev}) that the tangential component of the tension field is given by
\begin{gather*}
   \tau^{\tan}_{\vert \gamma(t)} =  \Bigl( \sum_{\mu=1}^{n} \nabla_{E^{\ast}_{\mu}} E^{\ast}_{\mu\vert\gamma(r(t))}\Bigr)^{\tan}
     - d\psi_{\vert \gamma(t)} \cdot \Bigl( \sum_{\mu=1}^{n} \nabla_{E^{\ast}_{\mu}}E^{\ast}_{\mu\vert \gamma(t)}\Bigr)^{\tan}.
\end{gather*}
Note that $\sum_{\mu=1}^n \nabla_{E^{\ast}_{\mu}} E^{\ast}_{\mu}$ and hence each summand in the formula above is independent of the choice of the orthonormal basis $e_1 = E^{\ast}_{1\vert\gamma(t)},\ldots,e_n = E^{\ast}_{n\vert\gamma(t)}$ of $T_{\gamma(t)}(G\cdot\gamma(t))$. Indeed, if $F_1,\ldots,F_n$ are other elements of $\mathfrak{n}$ such that $F^{\ast}_{1\vert\gamma(t)}, \ldots F^{\ast}_{n\vert\gamma(t)}$ is an orthonormal basis of $T_{\gamma(t)}(G\cdot\gamma(t))$ then $F_{\mu} = \sum_{\nu = 1}^n a_{\nu\mu} E_{\nu}$ where $(a_{\nu\mu})$ is an orthonormal matrix. Hence,
\begin{gather*}
  \sum_{\mu = 1}^n \nabla_{F^{\ast}_{\mu}} F^{\ast}_{\mu} = \sum_{\nu,\sigma = 1}^n \sum_{\mu = 1}^n a_{\nu\mu}a_{\sigma\mu} \nabla_{E^{\ast}_{\nu}} E^{\ast}_{\sigma} = \sum_{\nu = 1}^n \nabla_{E^{\ast}_{\nu}} E^{\ast}_{\nu}.
\end{gather*}

In order to actually compute the tangential component, we need the connection along the orbits. Using the skew-symmetry of $\nabla Z^{\ast}$ we get
\begin{multline}
  \langle \nabla_{X^{\ast}} X^{\ast}, Y^{\ast} \rangle = - \langle \nabla_{Y^{\ast}} X^{\ast},X^{\ast}\rangle
  = - \langle \nabla_{Y^{\ast}} X^{\ast},X^{\ast}\rangle + \langle \nabla_{X^{\ast}} Y^{\ast},X^{\ast}\rangle\\
  = \langle [X^{\ast},Y^{\ast}], X^{\ast}\rangle = -\langle [X,Y]^{\ast},X^{\ast}\rangle.
\label{direct}
\end{multline}
Hence,
\begin{multline}
  \langle \nabla_{X^{\ast}} X^{\ast}, Y^{\ast} \rangle_{\vert\gamma(t)} = -Q([X,Y], P_tX) = Q(Y,[X,P_t X]) \\
  = \langle (P_t^{-1}[X,P_t X])^{\ast}, Y^{\ast}\rangle_{\vert\gamma(t)}
\label{transform}
\end{multline}
by the skew-smmetry of $\ad_{X}$ and therefore
\begin{gather*}
  (\nabla_{X^{\ast}} X^{\ast}_{\vert\gamma(t)})^{\tan} = (P_t^{-1}[X,P_t X])^{\ast}_{\vert\gamma(t)}.
\end{gather*}
Note that $[X,P_t X]$ is contained in $\mathfrak{n}$ for all $X \in \mathfrak{n}$. Indeed, $Q(\,.\,,P_t\,.\,)$ is $\Ad_H$-invariant and hence
$\ad_Y$ is skew-symmetric with respect to this inner product for every $Y\in \ah$. The claim follows from
\begin{gather*}
  Q([X,P_t X],Y) = - Q(P_tX,[X,Y]) = Q(\ad_Y X, P_t X) = 0,
\end{gather*}
see \cite{gz}.

\begin{theorem}
\label{tanaltpart}
The tangential component of the tension field is given by
\begin{gather*}
  \tau^{\tan}_{\vert\gamma(t)} = 
  \bigl(P_{r(t)}^{-1} \sum_{\mu=1}^n [E_{\mu},P_{r(t)}E_{\mu}]\bigr)^{\ast}_{\vert\gamma(r(t))}
\end{gather*}
where $E_1,\ldots,E_n \in \mathfrak{n}$ are such that $E^{\ast}_{1\vert\gamma(t)},\ldots, E^{\ast}_{n\vert\gamma(t)}$ form an orthonormal basis of $T_{\gamma(t)}(G\cdot\gamma(t))$.
\label{tangentialone}
\end{theorem}

\begin{proof}
By the previous discussion we have
\begin{gather*}
  \tau^{\tan}_{\vert\gamma(t)} = \bigl(P_{r(t)}^{-1} \sum_{\mu=1}^n [E_{\mu},P_{r(t)}E_{\mu}]\bigr)^{\ast}_{\vert\gamma(r(t))}
  -  \bigl(P_{t}^{-1} \sum_{\mu=1}^n [E_{\mu},P_{t}E_{\mu}]\bigr)^{\ast}_{\vert\gamma(r(t))}.
\end{gather*}
Each of the two summands does not depend on the choice of the orthonormal basis. Since $P_t$ is symmetric with respect to $Q$ we can find a $Q$-orthonormal basis $F_1,\ldots F_n$ of eigenvectors of $P_t$ for just this single time $t$. The $F^{\ast}_{\mu\vert\gamma(t)}$ then form an orthogonal basis of $T_{\vert\gamma(t)}(G\cdot p)$. After rescaling, we get orthogonal eigenvectors of $E_{\mu} \in \mathfrak{n}$ of $P_t$, such that
$E^{\ast}_{1\vert\gamma(t)},\ldots, E^{\ast}_{n\vert\gamma(t)}$ form an orthonormal basis of $T_{\gamma(t)}(G\cdot\gamma(t))$. Hence, the second summand vanishes.
\end{proof}

The formula of Theorem\,\ref{tangentialone} is appropriate when the somewhat artifical use of the biinvariant metric $Q$ on the acting group and the endomorphism $P_{r(t)}$ might be justified by representation theoretic reasons. We finally provide an alternative formula for the tangential component of the tension field which avoids the objects $Q$ and $P_{r(t)}$ and only uses the Lie bracket from the acting group. For this formula we do not use formula (\ref{transform}) to evaluate the first term of equation (\ref{tensionev}), but rather stay with formula (\ref{direct}). In Section\,\ref{tangentialparts} we will see situations where Theorem\,\ref{tanaltpart} is more suitable and other situations where Theorem\,\ref{tanpart} is more suitable.
\begin{theorem}
\label{tanpart}
The tangential component of the tension field is given by
\begin{gather*}
  \tau^{\tan}_{\vert\gamma(t)} = -\sum_{\mu,\nu=1}^n \langle [E_{\mu},F_{\nu}]^{\ast},E_{\mu}^{\ast} \rangle_{\vert\gamma(r(t))} F^{\ast}_{\nu\vert\gamma(r(t))}.
\end{gather*}
where $E_1,\ldots,E_n \in \mathfrak{n}$ and $F_1,\ldots,F_n \in \mathfrak{n}$ are such that $E^{\ast}_{1\vert\gamma(t)},\ldots, E^{\ast}_{n\vert\gamma(t)}$ form an orthonormal basis of $T_{\gamma(t)}(G\cdot\gamma(t))$ and $F^{\ast}_{1\vert\gamma(r(t))},\ldots, F^{\ast}_{n\vert\gamma(r(t))}$ form an orthonormal basis of $T_{\gamma(r(t))}(G\cdot\gamma(r(t)))$.
\label{tangentialtwo}
\end{theorem}
This is Theorem\,B from the introduction.

\bigskip

\section{The normal components of the tension fields of the $(k,r)$-maps of spheres}
\label{spheres}
In this section we compute the tension fields of the $(k,r)$-maps for the $(g,m_0,m_1)$-actions on $ \Sph^{n+1}$.

\smallskip

For any $(g,m_0,m_1)$-action we consider an arbitrary fixed normal geodesic $\gamma$ as in Section\,\ref{kmaps}.
 In the extended group diagram we have the $2g = \abs{W}$ non-principal isotropy groups $K_0,\ldots,K_{2g-1}$ at the positions $\gamma(i\tfrac{\pi}{g})$ with Lie algebras $\ak_0,\ldots,\ak_{2g-1}$.
  Note that $K_{i+g} = K_g$ by the linearity of the action. 
  Let us denote the quotients $\ak_i/\ah$ by $\am_i$. Each $\am_i$ induces the $m_i$-dimensional vector space $\am^{\ast}_i$ of action fields that vanish at $\gamma(i\tfrac{\pi}{g})$. 
  These action fields are Jacobi fields along any geodesic, in particular, along our fixed normal geodesic $\gamma$. Since they vanish at $\gamma(i\tfrac{\pi}{g})$ they are of the form $\sin (t-i\tfrac{\pi}{g}) \cdot v(t)$ where $v$ is parallel along $\gamma$. The covariant derivative $\nabla_{\dot\gamma(t)} X^{\ast}$ of any such action field $X^{\ast}$ is $\cos(t-i\tfrac{\pi}{g}) \cdot v(t)$ and hence $\am^{\ast}_{i\vert\gamma(t)}$ is an eigenspace of the shape operator $S_{\vert\gamma(t)}$ to the eigenvalue $-\cot(t-i\tfrac{\pi}{g})$ for any regular time $t\neq \tfrac{\pi}{g}\Z$. Since $S_{\vert\gamma(t)}$ is symmetric with respect to the induced Riemannian metric on the orbit $G\cdot\gamma(t)$ the $\am^{\ast}_{0\vert\gamma(t)},\ldots,\am^{\ast}_{g-1\vert\gamma(t)}$ are pairwise orthogonal. Note that necessarily $\am_{i+g} = \am_i$. Just from counting dimensions it follows that the $\am_i$ span $\ag/\ah$. Hence, the Lie algebras $\ak_0,\ldots,\ak_{g-1}$ span $\ag$. This is a general property of the isotropy groups of cohomogeneity one action on spaces of positive curvature called linear primitivity in \cite{gwz}.
Each $\am^{\ast}_{i\vert\gamma(t)}$ is also an eigenspace of the endomorphism $J_t^{r(t)}$ defined in the previous section to the eigenvalue $\sin(r(t)-i\tfrac{\pi}{g}) / \sin(t-i\tfrac{\pi}{g})$. Hence, $(J_t^{r(t)})^{\ast} = J_t^{r(t)}$.

\begin{theorem}
For a $(g,m)$-action on $\Sph^{mg+1}$ the normal component of the tension field of a $(k,r)$-map is given by
\begin{multline*}
  2\sin^2 gt \cdot \tau^{\nor}_{\vert\gamma(t)} = 2\sin^2 gt \cdot \ddot r(t) + mg\sin 2gt \cdot \dot r(t) \\
  - mg \bigl( (g-1)\sin 2(r(t)-t) + \sin 2(r(t)+(g-1)t) \bigr).
\end{multline*}
\label{spherenormal}
\end{theorem}

\begin{proof}
Let $t \neq \tfrac{\pi}{g}\Z$ be any regular time. Theorem\,A yields
\begin{align*}
\tau^{\nor}_{\vert\gamma(t)} = \ddot r(t)+m\sum_{i=0}^{g-1}\cot(t-i\tfrac{\pi}{g})\dot r(t)
  -\tfrac{m}{2}\sum_{i=0}^{g-1}\frac{\sin 2(r(t)-i\tfrac{\pi}{g})}{\sin^2(t-i\tfrac{\pi}{g})}.
\end{align*}
The invariant sums can be evaluated with the standard cotangent identity $g\cot gt = \sum_{i=0}^{g-1} \cot(t-i\tfrac{\pi}{g})$ and Lemma\,\ref{2parac}.
\end{proof}

\begin{theorem}
For a $(g,m_0,m_1)$-action on $\Sph^{n+1}$ with $n = \frac{m_0+m_1}{2}g$ the normal component of the tension field of a $(k,r)$-map is given by
\begin{multline*}
  4\sin^{2} gt \cdot \tau^{\nor}_{\vert\gamma(t)} =
  4\sin^{2} gt \cdot \ddot r(t) + \bigl( g(m_0+m_1)\sin 2gt + 2g(m_0-m_1)\sin gt \bigr)\dot r(t)\\
  - g(g -2)\sin 2(r(t)-t) \bigl( m_0+m_1 + (m_0-m_1)\cos gt \bigr)\\
  -2g \sin\bigl(2(r(t)-t )+gt\bigr) \bigl( (m_0+m_1)\cos gt +m_0-m_1\bigr).
\end{multline*}
\label{spherenormaltwo}
\end{theorem}

\begin{proof}
For odd $g$ we have $m_0 = m_1 = m$. In this case the claimed formula can be seen to be equivalent to the formula of Theorem\,\ref{spherenormal}. If $g$ is even, Theorem\,A yields
\begin{align*}
  \tau^{\nor}_{\vert\gamma(t)} = \ddot r(t) + \sum_{i=0}^{g-1} m_i \cot(t-i\tfrac{\pi}{g})\dot r(t)
   -\frac{1}{2}\sum_{i=0}^{g-1} m_i \frac{\sin 2(r(t)-i\frac{\pi}{g})}{\sin^2(t-i\tfrac{\pi}{g})}.
\end{align*}
We set $h:=\tfrac{g}{2}$. Using the standard cotangent identity we obtain
\begin{multline*}
\sum_{i=0}^{g-1}m_i\cot(t-i\tfrac{\pi}{g}) =\sum_{\ell=0}^{h-1}m_0\cot(t-\ell \tfrac{\pi}{h})+\sum_{\ell=0}^{h-1}m_1\cot(t-\tfrac{\pi}{g}-\ell\tfrac{\pi}{h})\\
 =h(m_0\cot ht-m_1\tan ht) =\tfrac{g}{2}\bigl((m_0+m_1)\cot gt + (m_0-m_1)\tfrac{1}{\sin gt}\bigr).
\end{multline*}
Furthermore, from Lemma\,\ref{2parac} we obtain
\begin{multline*}
\sum_{i=0}^{g-1}m_i\frac{\sin(2r-i\frac{2\pi}{g})}{\sin^2(t-i\tfrac{\pi}{g})}
  = \sum_{\ell=0}^{h-1}m_0\frac{\sin(2r-\ell\frac{2\pi}{h})}{\sin^2(t-2\ell\tfrac{\pi}{g})}
   + \sum_{\ell=0}^{h-1}m_1\frac{\sin(2r-\tfrac{2\pi}{g}-\ell\tfrac{2\pi}{h})}{\sin^2(t-\tfrac{\pi}{g}-\ell\frac{\pi}{h})}\\
  = \tfrac{1}{2} g (g-2) \sin 2(r-t)\bigl( (m_0+m_1)+(m_0-m_1)\cos gt\bigr)\sin^{-2} gt\\
     +g\sin(2(r-t)+gt)\bigl((m_0+m_1)\cos gt +m_0-m_1\bigr)\sin^{-2} gt. \qedhere
\end{multline*}
\end{proof}

Theorem\,E from the introduction follows immediately from Theorem\,\ref{spherenormaltwo}.

\bigskip

\section{The normal components of the tension fields of the $(k,r)$-maps of orthogonal groups}
\label{groups}
In this section we compute the tension fields of the reparametrized $(k,r)$-maps $g\cdot \tilde\gamma(2t) \mapsto g\cdot \tilde\gamma(2r(t))$ for the $(g,m_0,m_1)$-actions on $\SO(n+2)$ with the metric $\tfrac{1}{2} \tr X\tp Y$. Note that the reparametrization here is done in a way so that the tension field expression fits systematically to the tension field expression for the actions on spheres. We determine the extended group diagrams of these actions and investigate the interplay between the action and the parallel transport along the normal geodesic. The approach is similar to that in the previous section but the concrete procedure is more complicated since the Jacobi operator along the normal geodesics has two eigenvalues instead of one.

\smallskip

We first determine the extended group diagrams. By conjugating $G$ by a suitable element of $\SO(n+2)$ we can assume that $\gamma(t) = \bmat \cos t\\ \sin t\\ 0\emat$ is a normal geodesic such that $\gamma(0)$ is contained in the singular orbit $N_0$. This normal geodesic $\gamma$ can be lifted horizontally to the normal geodesic
\begin{gather*}
  \tilde \gamma(t) = \bmat \cos t & -\sin t & 0\\ \sin t & \cos t & 0\\ 0 & 0 & \one_n\emat,
\end{gather*}
for the lifted action on $\SO(n+2)$.

\smallskip

The principal isotropy group $H$ of the action $G\times \Sph^{n+1} \to \Sph^{n+1}$ along $\gamma$ is a subgroup of $\SO(n) = \bmat 1 \\ & 1 \\ & & \ast\emat \subset \SO(n+2)$. In particular, all elements of $H$ commute with $\tilde\gamma(t)$ for all $t\in \R$.

\smallskip

Denote by $G_{\gamma(t)}$ the isotropy groups of the original action along the normal geodesic $\gamma$ and by $\tilde G_{\tilde\gamma(t)}$ the isotropy groups of the lifted action along the normal geodesic $\tilde\gamma$.

\begin{lemma}
The isotropy groups $\tilde G_{\tilde\gamma(t)}$ of the lifted action satisfy
\begin{gather*}
  \tilde G_{\tilde\gamma(t)} = \bigl\{ \bigl( A,\tilde\gamma(t)^{-1}A\tilde\gamma(t) \bigr) \,\bigl\vert\; A \in G_{\gamma(t)} \bigl\}.
\end{gather*}
In particular, each $\tilde G_{\tilde\gamma(t)}$ is isomorphic to the isotropy group $G_{\gamma(t)}$ of the original action for every $t\in \R$.
\end{lemma}
\begin{proof}
$\tilde\gamma(t) = A\tilde\gamma(t)\bmat 1 &\\ & B^{-1}\emat$ holds if and only if $A\in G_{\gamma(t)}$ and
$\bmat 1 & \\ & B \emat = \tilde\gamma(t)^{-1}A\tilde\gamma(t)$.
\end{proof}

\begin{corollary}
The principal isotropy group $\tilde H$ along $\tilde\gamma$ is the diagonal group $\Delta H = \{(h,h)\,\vert\,h\in H\}$ where $H$ is the principal isotropy group along $\gamma$.
\end{corollary}

The extended group diagram of any $(g,m_0,m_1)$-action on $\SO(n+2)$ hence contains the $2g$ groups $\tilde K_i = \tilde G_{\tilde\gamma(i\frac{\pi}{g})}$ with Lie algebras $\tilde\ak_i$. We denote the quotients $\tilde\ak_i/\tilde \ah_i$ by $\widetilde \am_i$. The goal of the next considerations is in particular to show that
the corresponding spaces $\widetilde\am_i^{\ast}$ of action fields are the eigenspaces of the shape operator $\tilde S$ of the principal orbits to the eigenvalues $-\frac{1}{2}\cot \frac{t-t_i}{2}$ where $t_i = i\frac{\pi}{g}$.

\smallskip

The parallel transport along $\tilde\gamma$ from time $0$ to any other time $t$ is given by multiplying an element of the Lie algebra $\aso(n+2)$ by $\tilde\gamma(t/2)$ simultaneously from the left and the right. In the language of symmetric spaces this isometry is called a transvection.
We now split the normal bundle of the geodesic $\tilde\gamma$ in $\SO(n+2)$ into two orthogonal parallel distributions. The first distribution contains the action fields $X^{\ast}$ with $X \in \{0\}\times \aso(n)$:
\begin{gather}
\label{son}
  \aso(n)^{\ast}_{\vert\tilde\gamma(t)} = \tilde\gamma(t)\cdot \aso(n) = \tilde\gamma(t/2)\cdot\aso(n)\cdot \tilde\gamma(t/2) = \aso(n).
\end{gather}
The second distribution is given by the parallel translates $V_{\vert\tilde\gamma(t)} = \tilde\gamma(t/2) \cdot V_{\vert\tilde\gamma(0)}\cdot \tilde\gamma(t/2)$ of
\begin{gather*}
  V_{\vert\tilde\gamma(0)} = \Bigl\{ \bmat 0 & 0 & -x\tp \\ 0 & 0 & -y\tp \\ x & y & 0\emat \;\Big\vert\;
  x,y \in \R^n \Bigr\}.
\end{gather*}

\begin{lemma}
\label{normalJacobi}
The normal Jacobi operator $R_{\dot{\tilde\gamma}(t)}$ along $\tilde\gamma$ has only two eigenspaces, namely, $\aso(n)^{\ast}_{\vert\tilde\gamma(t)}$ and $V_{\vert\tilde\gamma(t)}$. The corresponding eigenvalues are $0$ and $1/4$, respectively.
\end{lemma}

\begin{proof}
Because parallel transport is induced by isometries, it suffices to verify the statement at $t=0$. The Jacobi operator at time $0$ is given by $-\tfrac{1}{4}\ad_{\dot{\tilde\gamma}(0)}^2$ by standard results on the curvature tensor of compact Lie groups. It is straightforward to verify that $V_{\vert\tilde\gamma(0)}$ and $\aso(n)^{\ast}_{\vert\tilde\gamma(0)}$ are eigenspaces of $\ad_{\dot{\tilde\gamma}(0)}^2$ to the eigenvalues $-1$ and $0$, respectively.
\end{proof}

In order to keep the notation brief we denote the parallel translates of a vector $u(t_0) \in \aso(n)^{\ast}_{\vert\tilde\gamma(t_0)}$ along $\tilde \gamma$ simply by $u(t)$ and the parallel translates of a vector $v\in V_{\vert\tilde\gamma(t_0)}$ along $\tilde\gamma$ simply by $v(t)$.

\begin{corollary}
\label{jacobi}
Let $t_0 \in \R$ be an arbitrary time. The Jacobi field $Y$ along $\tilde\gamma(t)$ with initial data $Y(t_0) = u(t_0) + v(t_0)$ and $\tfrac{\nabla}{dt} Y(t_0) = u'(t_0) + v'(t_0)$ is given by
\begin{gather*}
  Y(t) = u(t) + (t-t_0)u'(t) + v(t)\cos\frac{t-t_0}{2} + 2v'(t) \sin \frac{t-t_0}{2}.
\end{gather*}
\end{corollary}

\begin{proof}
The vector field $Y$ given in the formula solves the Jacobi field equation $\frac{\nabla^2}{dt^2} Y(t) + R_{\dot{\tilde\gamma}(t)} Y(t) = 0$ and has the required inital values.
\end{proof}

Corollary\,\ref{jacobi} shows in particular that every parallel Jacobi field is an action field. Moreover,
for the restriction of action fields $X^{\ast}$ to the normal geodesic $\tilde\gamma$ the linear term on the right hand side of the formula in Corollary\,\ref{jacobi} has to vanish. Indeed, we have $X^{\ast}_{\vert\tilde\gamma(t)} = X^{\ast}_{\vert\tilde\gamma(t+2\pi)}$ because of the periodicity of the normal geodesic $\tilde\gamma$.

\begin{lemma}
Any action field $X^{\ast}$ with $X \in \widetilde \am_i$ is an eigenfield of the Jacobi operator to the eigenvalue $1/4$, i.e., $\widetilde \am_i^{\ast}\subset V$. For regular times $t\not\in\Z\cdot \frac{\pi}{g}$ the vector $X^{\ast}_{\vert\tilde\gamma(t)}$ is an eigenvector of the shape operator $\tilde S_{\vert\tilde\gamma(t)}$ of the principal orbit $\tilde G\cdot \tilde\gamma(t)$ to the eigenvalue $-\frac{1}{2}\cot \frac{t - t_i}{2}$ with $t_i = i\frac{\pi}{g}$.
\end{lemma}

\begin{proof}
The action field $X^{\ast}$ vanishes at $\gamma(t_i)$. Hence, along $\tilde\gamma$ it is of the form
\begin{gather*}
  X^{\ast}_{\vert\tilde\gamma(t)} = 2v'(t) \sin \frac{t-t_i}{2}
\end{gather*}
for some $v'\in V_{\vert\tilde\gamma(t_i)}$. This shows that $X^{\ast}_{\vert\tilde\gamma(t)}$ is an eigenvector of the Jacobi operator to the eigenvalue $\frac{1}{4}$. Moreover, we have
\begin{gather*}
  \tilde S_{\vert\tilde\gamma(t)}X^{\ast}_{\vert\tilde\gamma(t)} = -\nabla_{\dot{\tilde\gamma}(t)} X^{\ast}
    = -v'(t)\cos\frac{t - t_i}{2}.\qedhere
\end{gather*}
\end{proof}

\begin{corollary}
\label{orthogonalDecomp}
We have $V = \oplus_{i=0}^{2g-1} \widetilde\am_i^{\ast}$ where the sum is orthogonal.
\end{corollary}

\begin{proof}
Both spaces have the same dimension $2n$ and the right space is a subspace of the left one.
The $\widetilde\am_i^{\ast}$ are mutually orthogonal since they are eigenspaces of the shape operator to distinct eigenvalues.
\end{proof}

\begin{remark}
Usually there does not exist any biinvariant metric on $\tilde G = G \times \SO(n+1)$ such that the
orthogonal complements of $\tilde \ah$ in $\tilde \ak_i$ (which will by abuse of notation later also denoted by $\tilde \am_i$) are mutually orthogonal within the Lie algebra $\tilde \ag$.
\end{remark}

Note once again that for the $(g,m_0,m_1)$-actions on $\SO(n+2)$ in this section we use the reparametrization $\gamma(2t) \mapsto \tilde\gamma(2r(t))$ of the $(k,r)$-maps. Thus the formula of Theorem\,\ref{normalone} for the normal component of the tension field becomes
\begin{gather*}
  2\,\tau^{\nor}_{\vert\tilde\gamma(2t)} = \ddot r(t) - 2\dot r(t) \tr \tilde S_{\vert \tilde\gamma(2t)}
    + 2\tr\, (J_{2t}^{2r(t)})^{\ast} (\Pi_{2t}^{2r(t)})^{-1} \tilde S_{\vert \tilde\gamma(2r(t))} \Pi_{2t}^{2r(t)} J_{2t}^{2r(t)}.
\end{gather*}

\begin{lemma}
Each $\widetilde\am^{\ast}_{i\vert\tilde\gamma(2t)}$ is an eigenspace of the endomorphism $J_{2t}^{2r(t)}$ to the eigenvalue $\sin(r(t)-\frac{t_i}{2}) / \sin (t-\frac{t_i}{2})$. In particular, $(J_{2t}^{2r(t)})^{\ast} = J_{2t}^{2r(t)}$.
\end{lemma}

We now have gathered all the information to compute the normal component of the tension field.

\begin{theorem}
For a $(g,m)$-action on $\SO(mg+2)$ the normal component of the tension field $\tau$ of the reparametrized $(k,r)$-map $g\cdot\tilde\gamma(2t) \mapsto g\cdot\tilde\gamma(2r(t))$ is given by
\begin{multline*}
  4\sin^2 2gt \cdot \tau^{\nor}_{\vert\tilde\gamma(2t)} =
  2\sin^2 2gt \cdot \ddot r(t) + 2mg\sin 4gt \cdot \dot r(t) \\
  - 2mg \bigl( (2g-1)\sin 2(r(t)-t) + \sin 2(r(t)+(2g-1)t) \bigr).
\end{multline*}
\label{groupnormal}
\end{theorem}

\begin{proof}
Let $t \neq \Z\tfrac{\pi}{g}$ be any regular time. Evaluating the above formula yields
\begin{align*}
2\,\tau^{\nor}_{\vert\tilde\gamma(2t)} = \ddot r(t)+m\sum_{i=0}^{2g-1}\cot(t-i\tfrac{\pi}{2g})\dot r(t)
-\frac{m}{2}\sum_{i=0}^{2g-1}\frac{\sin 2(r(t)-i\frac{\pi}{2g})}{\sin^2(t-i\tfrac{\pi}{2g})}.
\end{align*}
The invariant sums can be evaluated with the standard cotangent identity $2g\cot 2gt = \sum_{i=0}^{2g-1} \cot(t-i\tfrac{\pi}{2g})$ and Lemma\,\ref{2parac}.
\end{proof}

\begin{theorem}
For a $(g,m_0,m_1)$-action on $\SO(n+2)$, $n = g\frac{m_0+m_1}{2}$ the normal component of the tension field of the reparametrized $(k,r)$-map $g\cdot\tilde\gamma(2t) \mapsto g\cdot\tilde\gamma(2r(t))$ is given by
\begin{multline*}
  8\sin^{2} 2gt \cdot \tau^{\nor}_{\vert\gamma(t)} =
  4\sin^{2} 2gt \cdot \ddot r(t) + \bigl( 2g(m_0+m_1)\sin 4gt + 4g(m_0-m_1)\sin 2gt \bigr)\dot r(t)\\
  - 2g(2g -2)\sin 2(r-t) \bigl( m_0+m_1 + (m_0-m_1)\cos 2gt \bigr)\\
  -4g \sin(2(r-t)+2gt) \bigl( (m_0+m_1)\cos 2gt +m_0-m_1\bigr).
\end{multline*}
\label{groupnormaltwo}
\end{theorem}

\begin{proof}
Analogous to the proof of Theorem\,\ref{spherenormaltwo}.
\end{proof}

Theorem\,F from the introduction follows immediately from Theorem\,\ref{groupnormaltwo}.

\bigskip

\section{Linear solutions of the $(g,m_0,m_1,k)$-boundary value problems}
\label{linearsolutions}
In Sections \ref{spheres} and \ref{groups} we computed the normal component of the tension field of the $(k,r)$-maps for any $(g,m_0,m_1)$-action on a sphere $\Sph^{n+1}$ and for any $(g,m_0,m_1)$-action on an orthogonal group $\SO(n+2)$ where $n = \frac{m_0+m_1}{2} g$. This normal component vanishes if
$r$ solves the $(g,m_0,m_1,k)$-boundary value problem or the $(2g,m_0,m_1,k)$-boundary value problem, respectively. Here, we determine when the linear function $r(t) = kt$ is a solution of these boundary value problems
and thus prove Lemma\,G from the introduction.

\begin{lemma}
For $m_0\neq m_1$ the linear solution $r(t)=kt$ with $k = jg+1$, $j\in\Z$, is a solution of the $(g,m_0,m_1,k)$-boundary value problem if and only if $k=1$ or $g = 2$ and $k=-1$.
\end{lemma}

\begin{proof}
The {\em if}-part is straightforward. In order to prove the {\em only if}-part we plug $r(t)=kt$ into the ODE of the $(g,m_0,m_1,k)$-boundary value problem and evaluate at $t = \tfrac{\pi}{2g}$. After straightforward algebraic manipulations we obtain the equation $k = \sin \frac{2j+1}{2}\pi$, i.e., $k = \pm 1$.
\end{proof}

\begin{lemma}
For $m_0= m_1 =: m$ the linear solution $r(t)=kt$ with $k = jg+1$, $j\in\Z$, is a solution of the $(g,m_0,m_1)$-BVP if and only if $j = 0$ or $j = -1$, i.e., $k= 1$ or $k = 1-g$.
\end{lemma}

\begin{proof}
Plugging $r(t) = kt$ into the ODE (\ref{ode}) yields
\begin{gather*}
  k \sin 2gt = (g-1) \sin 2(k-1)t + \sin 2(k-1+g)t.
\end{gather*}
It is straightforward to verify that this condition vanishes for $k=1$ and $k=1-g$. For general $k$ we evaluate the equation at $t_0 = \frac{\pi}{4g}$ and obtain
\begin{gather*}
  k = (g-1) \sin 2(k-1)t_0 + \sin 2(k-1+g)t_0.
\end{gather*}
Hence $\abs{k} \le g -1 + 1 = g$. This implies $j = 0$ or $j = -1$.
\end{proof}
The two lemmas above together are equivalent to Lemma\,G from the introduction.

\bigskip

\section{The tangential components of the tension fields}
\label{tangentialparts}
In this section we prove Theorem\,H from the introduction, i.e., we show that the tangential component of the tension field of any $(k,r)$-map for any $(g,m_0,m_1)$-action on $\Sph^{n+1}$ and on $\SO(n+2)$ vanishes except possibly for
\begin{gather*}
  (g,m_0,m_1) \in \{ (4,2,2\ell+1), (4,4,4\ell+3), (4,4,5), (4,6,9) \}.
\end{gather*}
We pursue two different strategies in the proof.

\smallskip

The first strategy is to employ Schur's lemma and finally apply Theorem\,\ref{tanaltpart}. This strategy works for the $(g,m_0,m_1)$-actions on $\Sph^{n+1}$ and $\SO(n+2)$ where the $\am_i$ in the decomposition
\begin{gather*}
  \ag = \ah \oplus \am_0 \oplus \am_1 \oplus \ldots \oplus \am_{g-1}
\end{gather*}
are inequivalent irreducible $H$-modules and where the $H$-module $\ah$ does not contain any irreducible submodules equivalent to some of the $\am_i$, i.e., it works for $g \le 3$, for $(g,m) = (4,2)$ and for $(g,m) = (6,2)$. Note that in this section we define $\am_i$ as the orthogonal complement of $\ah$ in $\ak_i$ with respect to the biinvariant metric $\frac{1}{2}\tr X\tp Y$ on $G \subset \SO(n+2)$, whereas previously we used the quotient definition $\am_i = \ak_i/\ah$.

\smallskip

The second, more general strategy is to determine the fixed point set of $H$ or $\tilde H$, respectively, and to employ the action of the Weyl group. This strategy works even without using the action of the Weyl group for $g \le 3$ if one employs all possible (even discrete) isometries that leave the foliation of the sphere invariant. In each of these cases the fixed point set of the principal isotropy group on the sphere $\Sph^{n+1}$ is just the unparametrized normal geodesic. In the cases where discrete isometries are available, i.e., for $g=1$, $g=2$ and $(g,m) = (3,2)$ the action has to be lifted to an action on $\OO(n+2)$ rather than on $\SO(n+2)$. This is no problem since the foliation of $\SO(n+2)$ by the orbits of an unextended $(g,m_0,m_1)$-action can be obtained from the foliation of $\OO(n+2)$ by the orbits of the extended $(g,m_0,m_1)$-action by intersecting the orbits with $\SO(n+2)$. The fixed point set of the principal isotropy group of the lifted action then consists of finitely many disjoint copies of normal geodesics, for example $2^{24}$ disjoint copies for $g = 3$, $m = 8$. Since we know by the first strategy that the tangential components of the tension fields of the $(k,r)$-maps for the $(g,m_0,m_1)$-actions with $g\le 3$ vanish, we do not provide any details about the computations of the fixed point sets in these cases. We rather work the second approach out in detail for $(g,m_0,m_1) = (4,m_0,1)$ and for $(g,m) = (6,1)$, since the first strategy does not work in these cases.

\smallskip

The tangential components of the tension fields of the $(k,r)$-maps for the remaining $(g,m_0,m_1)$-actions on $\Sph^{n+1}$ and $\SO(n+2)$ with
\begin{gather*}
  (g,m_0,m_1) \in \{(4,2,2\ell+1), (4,4,4\ell+3), (4,4,5), (4,6,9)\}
\end{gather*}
could be computed in an analogous way. We have avoided the lengthy computations and leave the question open whether the tangential components of the tension fields of all $(k,r)$-maps vanish in these cases. By Lemma\,G
from the introduction there are no linear solutions $r(t) = kt$ to the $(g,m_0,m_1,k)$-boundary value problem except for $k = 1$ in these cases anyway.

\subsection{Using Schur's Lemma}
The goal of this subsection is to establish the following result.

\begin{theorem}
\label{Schur}
Let $G \times \Sph^{n+1} \to \Sph^{n+1}$ be a $(g,m_0,m_1)$-action such that the $\am_i$ in the decomposition
\begin{gather}
\label{spheredecomp}
  \ag = \ah \oplus \am_0 \oplus \am_1 \oplus \ldots \oplus \am_{g-1}
\end{gather}
are inequivalent irreducible $H$-modules and such that the $H$-module $\ah$ does not contain any irreducible submodules equivalent to some of the $\am_i$. Then the tangential component of the tension field of any $(k,r)$-map for the $(g,m_0,m_1)$-action on $\Sph^{n+1}$ and on $\SO(n+2)$ vanishes.
\end{theorem}

Before we turn to the proof of this theorem we show for which triples $(g,m_0,m_1)$ the hypothesis holds.

\begin{lemma}
The hypothesis of Theorem\,\ref{Schur} holds for $g \le 3$ and for $(g,m_0,m_1) \in \{(4,2,2),(6,2,2)\}$.
\end{lemma}

\begin{proof}
For the $(1,m)$-action on $\Sph^{m+1}$ the principal orbit is $\SO(m+1)/\SO(m)$ and the Lie algebra $\aso(m+1)$ decomposes as an $\SO(m)$-module into the adjoint representation of $\SO(m)$ and the irreducible standard representation on $\R^m$.

Similarly, the principal orbit of the $(2,m_0,m_1)$-action is
\begin{gather*}
  \SO(m_0+1)\times\SO(m_1+1)/\SO(m_0)\times \SO(m_1)
\end{gather*}
and the Lie algebra $\aso(m_0+1) \oplus \aso(m_1+1)$ decomposes as an $\SO(m_0) \times \SO(m_1)$-module into the adjoint representations of $\SO(m_0)$ and $\SO(m_1)$ and the irreducible standard representations of $\SO(m_0)$ and $\SO(m_1)$.

For $g = 3$ the principal orbits are diffeomorphic to the manifolds of flags in the projective planes $\RP^2$, $\CP^2$, $\HP^2$, and $\OP^2$, i.e., diffeomorphic to
\begin{gather*}
  \SO(3)/\Z_2\times \Z_2, \quad \SU(3)/T^2, \quad \Syp(3)/\Syp(1)^3, \quad \text{and} \quad \Ffour/\Spin(8).
\end{gather*}
It is well-known that the isotropy representation of each of these spaces splits into three inequivalent $m$-dimensional $H$-modules all of which are inequivalent to the submodules of $\ah$. These decompositions of $\ag$ given by the isotropy representation are precisely the decompositions (\ref{spheredecomp}) of $\ag$ by Schur's Lemma.

The adjoint actions of the compact Lie groups $\Syp(2)$ and $\Gtwo$ provide the $(4,2)$-action on $\Sph^9$ and the $(6,2)$-action on $\Sph^{13}$. In each case decomposition (\ref{spheredecomp}) is the root space decomposition of the Lie group, which clearly has the desired property.
\end{proof}

The proof of Theorem\,\ref{Schur} is simple for the actions on the spheres. Indeed, by Schur's Lemma, the $H$-equivariant endomorphism $P_{t}$ is diagonal for any regular time $t \not \in \frac{\pi}{g}\Z$. Hence, the vanishing of the tangential component follows from Theorem\,\ref{tangentialone} by choosing an orthonormal basis compatible with the decomposition (\ref{spheredecomp}).

\smallskip

The proof of Theorem\,\ref{Schur} for the lifted actions is given in the rest of this subjection. We first need to establish some more facts about the actions on the spheres. As before, let $G \times \Sph^{n+1} \to \Sph^{n+1}$ be a $(g,m_0,m_1)$-action with normal geodesic $\gamma(t) = \bmat \cos t\\ \sin t\\ 0\emat$. The non-regular isotropy groups $K_i$ appear at the points $\gamma(t_i)$ with $t_i = i \frac{\pi}{g}$. The principal isotropy group $H$ along $\gamma$ is constant and equal to the isotropy groups $K_{i\vert\dot\gamma(t_i)}$ of the actions of $K_i$ on the normal spaces to the orbits $G\cdot\gamma(t_i)$ at $\gamma(t_i)$. The orbit $K_i\cdot \dot\gamma(t_i)$ is a linear and hence totally geodesic subsphere $S_i^{m_i}$ of $\Sph^{n+1}$. We endow $G \subset \SO(n+2)$ with the biinvariant metric $\frac{1}{2}\tr X\tp Y$. The orthogonal complements $\am_i$ of the Lie algebra $\ah$ of $H$ in the Liealgebras $\ak_i$ of $K_i$ are irreducible $H$-modules by assumption. Hence, there exists up to a constant factor just one homogeneous metric on $K_i/H$. The map $K_i \to S_i^{m_i}$, $k\mapsto k\cdot \dot\gamma(t_i)$ is up to a constant factor a Riemannian submersion from $K_i$ with the biinvariant metric inherited from $G$ to $S_i^{m_i}$ with the standard metric inherited from $\Sph^{n+1}$. Let $e_{i,1},\ldots, e_{i,m_i}$ denote an orthonormal basis of $T_{\dot \gamma(t_i)} S_i^{m_i}$ and $E_{i,1}, \ldots, E_{i,m_i}$ denote their horizontal lifts to $\am_i$ at the unit element $\one$.

\begin{lemma}
The curves $\exp s E_{i,\mu}\cdot \dot \gamma(t_i)$ are all great circles in $S_i^{m_i}$.
\end{lemma}

\begin{proof}
The horizontal lifts of the great circles $\dot \gamma(t_i)\cos s + e_{i,\mu} \sin s$ through the unit element $\one$ of $K_i$ are horizontal geodesics in $K_i$ and hence $1$-parameter subgroups of $K_i$. The $E_{i,1}, \ldots, E_{i,m_i}$
are the initial vectors of these $1$-parameter subgroups.
\end{proof}

\begin{corollary}
\label{specialform}
As an element of $\aso(n+2)$ each $E_{i,\mu} \subset \am_i$ is of the form
\begin{gather*}
  E_{i,\mu} = \bmat 0 & 0 & v\tp \sin t_i \\ 0 & 0 & -v\tp \cos t_i \\ -v \sin t_i & v \cos t_i & Y \emat
\end{gather*}
with $Y v = 0$. In particular, $E_{i,\mu}$ commutes with its projection to
$\aso(n) = \bmat 0 & 0 & 0\\ 0 & 0 & 0 \\ 0 & 0 & \ast \emat$.
\end{corollary}

\begin{proof}
We have
\begin{gather*}
  \exp sE_{i,\mu} \cdot \gamma(t_i) = \gamma(t_i), \\
  \exp sE_{i,\mu} \cdot \dot \gamma(t_i) = \dot \gamma(t_i) \cos s + e_{i,\mu} \sin s, \\
  \exp sE_{i,\mu} \cdot e_{i,\mu} = -\dot \gamma(t_i) \sin s + e_{i,\mu} \cos s.
\end{gather*}
Hence,
\begin{gather*}
  \exp sE_{i,\mu} = \bmat \cos t_i & -\sin t_i & 0 \\ \sin t_i & \cos t_i & 0 \\ 0 & 0 & A\emat
  \bmat 1 & 0 & 0 & 0 \\ 0 & \cos s & -\sin s & 0 \\ 0 & \sin s & \cos s & 0 \\ 0 & 0 & 0 & \ast \emat
  \bmat \cos t_i & \sin t_i & 0 \\ -\sin t_i & \cos t_i & 0 \\ 0 & 0 & A^{-1}\emat
\end{gather*}
where $A\in \SO(n)$ is any matrix whose first column is the projection of $e_{i,\mu}\subset S_i^{m_i} \subset \Sph^{n+1} \subset \R^{n+2}$ to the last $n$ components (the first two are $0$).  Differentiating with respect to $s$ and evaluating at $s = 0$ yields the claimed statement.
\end{proof}

We can now turn to the lifted $(g,m_0,m_1)$-action
\begin{gather*}
  G \times \SO(n+1) \times \SO(n+2) \to \SO(n+2), \quad (A,B) \cdot C = ACB^{-1}.
\end{gather*}
We endow $\SO(n+2)$ with the biinvariant metric $\tfrac{1}{2}\tr X\tp Y$ as before. The group $\tilde G = G \times \SO(n+1)$ is considered to be the Riemannian product of $G \subset \SO(n+2)$ with $\tfrac{1}{2}\tr X\tp Y$ and $\SO(n+1)$ with $\tfrac{1}{2}\tr X\tp Y$. Note that in this section $\tilde \am_i$ is defined to be the orthogonal complement of $\tilde \ah$ in $\tilde \ak_i$ whereas it was previously defined to be the quotient $\tilde \ak_i/\tilde \ah$.

\smallskip

Our goal is to construct an orthogonal decomposition
\begin{gather}
\label{abeliandecomp}
  \tilde \ag = \tilde \ah \oplus \tilde \ab_1 \oplus \tilde \ab_2 \oplus \ldots \oplus \tilde \ab_{n} \oplus \tilde \ah_{\pm} \oplus \tilde \aq
\end{gather}
of the Lie algebra of $\tilde G = G \times \SO(n+1)$ such that for each regular time $t \not \in \frac{\pi}{g}\Z$ the sum
\begin{gather}
\label{abelianstardecomp}
   \tilde \ab^{\ast}_{1\vert\tilde\gamma(t)} \oplus \tilde \ab^{\ast}_{2\vert\tilde\gamma(t)} \oplus \ldots \oplus \tilde \ab^{\ast}_{n\vert\tilde\gamma(t)} \oplus \tilde \ah_{\pm\vert\tilde\gamma(t)}^{\ast} \oplus \tilde \aq^{\ast}_{\vert\tilde\gamma(t)}
\end{gather}
is orthogonal, each $\tilde \ab_{\mu}$ is a $\tilde P_t$-invariant abelian subalgebra of $\tilde \ag$, and $\ah_{\pm}^{\ast}$ and $\tilde \aq$ are eigenspaces of $\tilde P_{t}$. Theorem\,\ref{tanpart} will then imply that the tangential components of the tension fields of all $(k,r)$-maps vanish.

\smallskip

We use the canonical identification $H \to \tilde H$, $h \mapsto (h,h)$ and introduce several $H$-equivariant maps.
First, for any $i \in \{0,\ldots, g-1\}$ the map
\begin{gather*}
  \iota: \am_i \to \tilde\am_i, \quad X \mapsto \bigl(X, \tilde\gamma(t_i)^{-1}X\gamma(t_i) \bigr),
\end{gather*}
where $t_i = i\frac{\pi}{g}$, is $H$-equivariant since all elements of $H$ commute with $\tilde\gamma(t)$ for all $t\in \R$.  For the same reason the maps
\begin{gather*}
  \sigma : \tilde\am_i \to \tilde\am_{i+g}, \quad (X, \hat X) \mapsto \Bigl(X, \bmat -\one_2 & 0 \\ 0 & \one_{n} \emat \hat X \bmat -\one_2 & 0 \\ 0 & \one_{n} \emat \Bigr)
\end{gather*}
and
\begin{gather*}
  \pi : \tilde\am_i \to \{0\} \times \aso(n), \quad (X,\hat X) \mapsto \Bigl( 0, \tfrac{1}{2} \bigl(\hat X + \bmat -\one_2 & 0 \\ 0 & \one_{n} \emat \hat X \bmat -\one_2 & 0 \\ 0 & \one_{n} \emat \bigr)\Bigr).
\end{gather*}
are $H$-equivariant. Note that $\pi$ is just the projection of $X$ to its $\aso(n)$-part blown up a little for formal reasons. We denote the image of $\pi$ by $\tilde \ap_i$.
Obviously, we have $\pi\circ \sigma = \pi$, i.e., the projections of $\tilde\am_i$ and $\tilde\am_{i+g}$ are the same.
{\em In the following we assume that each $\tilde \ap_i$ contains a non-zero element and is thus isomorphic to $\tilde \am_i$ by Schur's lemma.} It is obvious how to modify the following arguments if some of the $\tilde \ap_i = \{0\}$.

Next, we set $\tilde\ah_{\pm} = \{(X,-X) \;\vert\; X \in \ah\}$. Note that for $X\in\ah$,
\begin{gather*}
  (X,-X) - (X,X) = (0,-2X).
\end{gather*}
Hence, $(X,-X)^{\ast}_{\vert\tilde\gamma(t)} = (0,-2X)^{\ast}_{\vert\tilde\gamma(t)}$ and
$\tilde \ah^{\ast}_{\pm\vert\tilde\gamma(t)} = (\{0\} \times \ah)^{\ast}_{\vert\tilde\gamma(t)}$.

Finally, we define the space $\tilde \aq$ to be the orthogonal component of
\begin{gather*}
  \tilde \ap_0 \oplus \ldots \oplus\tilde \ap_{g-1}\oplus \{0\}\times\ah
\end{gather*}
in $\{0\} \times \aso(n)$. 

\begin{lemma}
For regular times $t\not\in \frac{\pi}{g}\Z$, the decomposition
\begin{gather*}
  \tilde\ag^{\ast}_{\vert\tilde\gamma(t)} = \tilde \am^{\ast}_{0\vert\tilde\gamma(t)} \oplus \ldots \tilde \am^{\ast}_{2g-1\vert\tilde\gamma(t)} \oplus \tilde\ap^{\ast}_{0\vert\tilde\gamma(t)}\oplus \ldots \tilde \ap^{\ast}_{g-1\vert\tilde\gamma(t)}
  \oplus \tilde \ah^{\ast}_{\pm\vert\tilde\gamma(t)} \oplus \tilde \aq^{\ast}_{\vert\tilde\gamma(t)}
\end{gather*}
is orthogonal.
\end{lemma}

\begin{proof}
The orthogonality of the decomposition
\begin{gather*}
  \tilde\ag^{\ast}_{\vert\tilde\gamma(t)} = \tilde \am^{\ast}_{0\vert\tilde\gamma(t)} \oplus \ldots \oplus \tilde \am^{\ast}_{2g-1\vert\tilde\gamma(t)} \oplus \aso(n)^{\ast}_{\vert\tilde\gamma(t)}
\end{gather*}
was established in Section\,\ref{groups}. Since the $H$-module $\ah$ does not contain any $H$-submodules equivalent to any of the $\am_i$, the $H$-module $\tilde \ah^{\ast}_{\pm\vert\tilde\gamma(t)}$ is perpendicular to
\begin{gather*}
  \tilde \am^{\ast}_{0\vert\tilde\gamma(t)} \oplus \ldots \oplus \tilde \am^{\ast}_{2g-1\vert\tilde\gamma(t)} \oplus \tilde\ap^{\ast}_{0\vert\tilde\gamma(t)}\oplus \ldots \tilde \ap^{\ast}_{g-1\vert\tilde\gamma(t)}.
\end{gather*}
Finally, $\tilde\aq^{\ast}_{\vert\tilde\gamma(t)}$ is perpendicular to all other summands by construction, since $\aso(n) \to \aso(n)^{\ast}_{\vert\gamma(t)}$ is an isometry for all times $t \in \R$.
\end{proof}

Now consider the elements
\begin{gather*}
  \tilde E_{i,\mu} = \iota (E_{i,\mu}) \in \tilde \am_i, \quad
  \sigma(\tilde E_{i,\mu}) \in \tilde \am_{i+g}, \quad \text{and}\quad
  \pi(\tilde E_{i,\mu}) \in \tilde \ap_i.
\end{gather*}
for all $i \in \{0,\ldots, g-1\}$ and $\mu \in \{1,\ldots,m_i\}$.

\begin{lemma}
The set of all $\tilde E_{i,\mu\vert\tilde\gamma(t)}^{\ast}$, $\sigma(\tilde E_{i,\mu\vert\tilde\gamma(t)})^{\ast}$, and $\pi(\tilde E_{i,\mu\vert\tilde\gamma(t)})^{\ast}$ is an orthogonal basis of
\begin{gather*}
  \tilde \am^{\ast}_{0\vert\tilde\gamma(t)} \oplus \ldots \tilde \am^{\ast}_{2g-1\vert\tilde\gamma(t)} \oplus \tilde\ap^{\ast}_{0\vert\tilde\gamma(t)}\oplus \ldots \tilde \ap^{\ast}_{g-1\vert\tilde\gamma(t)}
\end{gather*}
for any regular time $t \not \in\frac{\pi}{g}\Z$.
\end{lemma}

\begin{proof}
Each $\am_i$ is irreducible. Therefore, the $H$-equivariant isomorphisms $\iota: \am_i \to \tilde \am_i$, $\sigma: \tilde \am_i\to \tilde \am_{i+g}$, $\pi:\tilde \am_i \to \ap_i$, $\tilde \am_i \to \tilde\am_{i\vert\tilde\gamma(t)}^{\ast}$, and $\tilde \ap_i \to \tilde \ap_{i\vert\tilde\gamma(t)}^{\ast}$ identify the inner products on all these spaces up to scalar factors and hence preserve perpendicularity.
\end{proof}

Note that the decomposition
\begin{alignat}{6}
  \tilde \ag &= \tilde \ah &&\oplus \tilde \am_0 &&\oplus \tilde \am_1 &&\oplus \ldots &&\oplus \tilde \am_{g-1} && \notag \\
  & &&\oplus \tilde \am_g &&\oplus \tilde \am_{g+1} &&\oplus \ldots &&\oplus \tilde \am_{2g-1} && \label{groupdecomp}
\\
  & &&\oplus \tilde \ap_0 &&\oplus \tilde \ap_1 &&\oplus \ldots &&\oplus \tilde \ap_{g-1} &&\oplus \tilde \ah_{\pm} \oplus \tilde \aq \notag
\end{alignat}
itself is not necessarily orthogonal.

\begin{lemma}
The columns in the decomposition $(\ref{groupdecomp})$ are mutually perpendicular.
\label{columndecomp}
\end{lemma}

\begin{proof}
For each $i\in \{0,\ldots,g-1\}$ the summands $\tilde \am_i$, $\tilde \am_{i+g}$ and $\tilde \ap_i$ in the decomposition (\ref{groupdecomp}) are equivalent $H$-modules. They are inequivalent to any other $\tilde \am_j$, $\tilde \am_{j+g}$, $\tilde \ap_j$. Since $\ah$ does not contain any $H$-invariant subspaces equivalent to one of the $\am_i$, the space $\tilde \ah_{\pm}$ is perpendicular to the $\tilde \am_i$, $\tilde \am_{i+g}$ and $\tilde \ap_i$. Finally, $\tilde \aq$ is perpendicular to all of the other spaces by its definition.
\end{proof}

Now set $\tilde E_{i,\mu} = \iota (E_{i,\mu})$ and
\begin{gather*}
  \tilde \ab_{i,\mu} = \spann\{ \tilde E_{i,\mu}, \sigma(\tilde E_{i,\mu}), \pi(\tilde E_{i,\mu}) \}.
\end{gather*}

\begin{lemma}
The sum $\bigoplus_{i=0}^{g-1} \bigoplus_{\mu=1}^{m_i} \tilde\ab_{i,\mu}$ is orthogonal.
\end{lemma}

\begin{proof}
For any $i\neq j$, any $\tilde \ab_{i,\mu}$ is perpendicular to any $\tilde \ab_{j,\nu}$ by Lemma\,\ref{columndecomp}. We now consider the case $i = j$ and $\mu \neq \nu$.
By Corollary\,\ref{specialform}, we have
\begin{gather*}
  E_{i,\mu} = \bmat 0 & 0 & v\tp \sin t_i \\ 0 & 0 & -v\tp \cos t_i \\ -v \sin t_i & v \cos t_i & Y \emat
  \quad \text{and} \quad
  E_{i,\nu} = \bmat 0 & 0 & w\tp \sin t_i \\ 0 & 0 & -w\tp \cos t_i \\ -w \sin t_i & w \cos t_i & Z \emat
\end{gather*}
for some $Y$ and $v$ with $Yv = 0$ and some $Z$ and $w$ with $Zw = 0$. Now, we have
\begin{gather*}
  \tilde E_{i,\mu} = \bigl( E_{i,\mu}, \bmat 0 & 0 & 0 \\ 0 & 0 & -v\tp \\ 0 & v & Y \emat \bigr), \quad 
  \sigma(\tilde E_{i,\mu}) = \bigl( E_{i,\mu}, \bmat 0 & 0 & 0 \\ 0 & 0 & v\tp \\ 0 & -v & Y \emat\bigr), \quad
  \pi(\tilde E_{i,\mu}) = \bigl( 0, \bmat 0 & 0 & 0 \\ 0 & 0 & 0\\ 0 & 0 & Y \emat \bigr)
\end{gather*}
and
\begin{gather*}
  \tilde E_{i,\nu} = \bigl( E_{i,\nu}, \bmat 0 & 0 & 0 \\ 0 & 0 & -w\tp \\ 0 & w & Z \emat \bigr), \quad 
  \sigma(\tilde E_{i,\nu}) = \bigl( E_{i,\nu}, \bmat 0 & 0 & 0 \\ 0 & 0 & w\tp \\ 0 & -w & Z \emat\bigr), \quad
  \pi(\tilde E_{i,\nu}) = \bigl( 0, \bmat 0 & 0 & 0 \\ 0 & 0 & 0\\ 0 & 0 & Z \emat \bigr).
\end{gather*}
As noted earlier, the $H$-equivariant isomorphisms $\iota$, $\sigma$, and $\pi$ preserve perpendicularity.
Hence, $Y$ and $Z$ are perpendicular. But then $v$ and $w$ are perpendicular as well. It is now immediate from the formulas above that $\tilde \ab_{i,\mu}$ and $\tilde \ab_{i,\nu}$ are perpendicular.
\end{proof}

\begin{lemma}
Each $\tilde \ab_{i,\mu}$ is an abelian subalgebra of $\tilde \ag$.
\end{lemma}

\begin{proof}
This is an immediate consequence of Corollary\,\ref{specialform}.
\end{proof}

\begin{lemma}
For any regular time $t$, each $\tilde\ab_{i,\mu}$ is $\tilde P_{t}$-invariant and $\tilde \ah_{\pm}$ and $\tilde \aq$ are eigenspaces of $\tilde P_t$.
\end{lemma}

\begin{proof}
Let $\tilde Q$ denote the biinvariant metric on $\tilde G = G \times \SO(n+1)$ chosen at the beginning of this section, i.e., the product metric of $\frac{1}{2} \tr X\tp Y$ on $G \subset \SO(n+2)$ with $\frac{1}{2}\tr X\tp Y$ on $\SO(n+1)$.
The $\tilde \ah_{\pm}$, and $\tilde \aq$ and all the $\tilde a_{i,\mu}$ are mutually orthogonal and their respective action spaces at any regular time are mutally orthogonal. Hence, it follows from the equation
\begin{gather*}
  \tilde Q(\tilde P_t X, Y ) = \langle X^{\ast}_{\vert\tilde\gamma(t)}, Y^{\ast}_{\vert\tilde\gamma(t)} \rangle
\end{gather*}
that the $\tilde \ah_{\pm}$, $\tilde \aq$ and each of the $\tilde a_{i,\mu}$ is $\tilde P_t$-invariant. Since $\{0\}\times \aso(n) \to \aso(n)^{\ast}_{\vert\tilde\gamma(t)}$ is an isometry for all regular times, it follows that $\tilde P_t$ is the identity on $\tilde \aq$ and that $\tilde \ah_{\pm}$ is an eigenspace to the eigenvalue $2$ of $\tilde P_t$.
\end{proof}
Theorem\,\ref{Schur} now follows from Theorem\,\ref{tanpart}.

\medskip

\subsection{The case $(g,m_0,m_1) = (4,m_0,1)$}
Let $M_{m_0+2,2}$ denote the space of real $(m_0+2)\times 2$ matrices with the norm $\abs{X}^2 = \tr X\tp X$. We consider the action
\begin{gather*}
  \OO(m_0+2) \times \OO(2) \times M_{m_0+2,2} \to M_{m_0+2,2},\quad
  \bigl( (A,B), X\bigr) \mapsto AXB^{-1}.
\end{gather*}
This action induces a $(4,m_0,1)$-action on the unit sphere in $M_{m_0+2,2}$. A normal geodesic for this action is
\begin{gather*}
  \gamma(t) = \bmat \cos t & 0 \\ 0 & \sin t \\ 0 & 0 \emat
\end{gather*}
where each zero in the last row is $0 \in \R^{m_0}$. The principal isotropy group along $\gamma$ is
\begin{gather*}
  H = \Bigl\{ \Bigl( \bmat \epsilon_1 & & \\ & \epsilon_2 & \\ & & C \emat, \bmat \epsilon_1 & \\ & \epsilon_2 \emat \Bigr) \;\Big\vert\, \epsilon_1, \epsilon_2 = \pm 1, C \in \OO(m_0)\Bigr\}.
\end{gather*}
Non-principal isotropy groups along $\gamma$ appear at the multiples of $\frac{\pi}{4}$:
\begin{xalignat*}{2}
  t &= 0 \mod \pi: & K_0 &= \Bigl\{ \Bigl( \bmat \epsilon_1 & \\ & A' \emat, \bmat \epsilon_1 & \\ & \epsilon_2 \emat \Bigr) \;\Big\vert\, \epsilon_1, \epsilon_2 = \pm 1, A' \in \OO(m_0+1)\Bigr\},\\
  t &= \tfrac{\pi}{4} \mod \pi: & K_1 &= \bigl\{ \bigl( \bmat B & \\  & C \emat, B \bigr) \;\big\vert\, B\in \OO(2), C \in \OO(m_0)\bigr\},\\
  t &= \tfrac{\pi}{2} \mod \pi: & K_2 &= \Bigl\{ \Bigl( \bmat \ast & & \ast \\ & \epsilon_2 & \\ \ast & & \ast \emat, \bmat \epsilon_1 & \\ & \epsilon_2 \emat \Bigr) \;\Big\vert\, \epsilon_1, \epsilon_2 = \pm 1\Bigr\},\\
  t &= \tfrac{3\pi}{4} \mod \pi: & K_3 &= \bigl\{ \bigl( \bmat SBS & \\  & C \emat, B \bigr) \;\big\vert\, B\in \OO(2), C \in \OO(m_0), S = \bmat 1 & \\ & -1\emat\bigr\}.
\end{xalignat*}

Via the above action $\OO(m_0+2) \times \OO(2)$ naturally becomes a subgroup of $\OO(M_{m_0+2,2})$. From now on we identify the orthogonal group $\OO(M_{m_0+2,2})$ with $\OO(2m_0+4)$ by using the basis
\begin{multline*}
  \bmat 1 & 0 \\ 0 & 0\\ 0 & 0 \emat,
  \bmat 0 & 0 \\ 0 & 1\\ 0 & 0 \emat,
  \bmat 0 & 0 \\ 1 & 0\\ 0 & 0 \emat,
  \bmat 0 & 1 \\ 0 & 0\\ 0 & 0 \emat,
  \bmat 0 & 0 \\ 0 & 0 \\ 0 & b_1 \emat, \ldots, \bmat 0 & 0 \\ 0 & 0 \\ 0 & b_{m_0} \emat,
  \bmat 0 & 0 \\ 0 & 0 \\ b_1 & 0 \emat, \ldots, \bmat 0 & 0 \\ 0 & 0 \\ b_{m_0} & 0 \emat
\end{multline*}
of $M_{m_0+2,2}$ where $b_1, \ldots, b_{m_0}$ denotes the standard basis of $\R^{m_0}$. It is straightforward to compute that the homomorphism
\begin{gather*}
  \Theta: \OO(m_0+2) \times \OO(2) \to \OO(2m_0+4)
\end{gather*}
maps the elements of the principal isotropy group $H$ as follows:
\begin{gather}
\label{isotropyfourone}
  \Bigl( \bmat \epsilon_1 & & \\ & \epsilon_2 & \\ & & C \emat, \bmat \epsilon_1 & \\ & \epsilon_2 \emat \Bigr) \mapsto
    \bmat \one_2 & & & \\ & \epsilon_1\epsilon_2 \one_2 & & \\ & & \epsilon_2 C & \\ & & & \epsilon_1 C \emat.
\end{gather}
The following statement is now evident.
\begin{lemma}
The fixed point set $(\Sph^{2m_0+3})^H$ just consists of the unparametrized normal geodesic $\gamma(\R)$.
\end{lemma}

\begin{corollary}
The tangential part of the tension field of any $(k,r)$-map vanishes for the $\OO(m_0+2)\times \OO(2)$-action on $\Sph^{2m_0+3}$.
\end{corollary}

\smallskip

The $\OO(m_0+2) \times \OO(2)$-action on $\Sph^{2m_0+3}$ lifts to the $\OO(m_0+2) \times \OO(2) \times \OO(2m_0+3)$-action on $\OO(2m_0+4)$ given by
\begin{gather*}
  (A,B,C) \cdot D = \Theta(A,B)DC^{-1}.
\end{gather*}
Note that we use the metric $\langle X,Y\rangle = \tfrac{1}{2} \tr X\tp Y$ on $\OO(2m_0+4)$. A normal geodesic is
\begin{gather*}
  \tilde \gamma(t) = \bmat \cos t & -\sin t & 0\\ \sin t & \cos t & 0\\ 0 & 0 & \one_{2m_0+2}\emat.
\end{gather*}
with principal isotropy group $\tilde H = \Delta H$.

\begin{lemma}
The fixed point set of $\tilde H$ in the tangent space $T_{\tilde\gamma(t)} \tilde G \cdot \tilde \gamma(t)$ is generated by $\tilde F^{\ast}_{\vert\tilde\gamma(t)}$ along where
\begin{gather*}
  \tilde F = \bmat  0 & -1 & & \\ 1 & 0 & & \\ & & 0_{m_0} & \\ & & & 0_{m_0} \emat \in \aso(2m_0+2) \subset \aso(2m_0+4).
\end{gather*}
\end{lemma}

Note that $\tilde F^{\ast}$ is a parallel vector field along $\tilde \gamma$.

\begin{proof}
Using the fact that the center of $\OO(m_0)$ is $\pm \one$ it is straightforward to verify that the fixed point set of
$\tilde H$ in $\OO(2m_0+4)$ consist of elements of the form
\begin{gather*}
  \bmat \ast_2 & & & \\ & \ast_2 & & \\ & & \pm \one_{m_0} & \\ & & & \pm \one_{m_0} \emat.
\end{gather*}
The tangent space to this fixed point set at $\tilde\gamma(t)$ is spanned by $\dot{\tilde\gamma}(t)$ and $\tilde F^{\ast}_{\vert\tilde\gamma(t)}$.
\end{proof}

The Weyl group of the $\OO(m_0+2) \times \OO(2)$-action on $\Sph^{2m_0+3}$ is generated by
\begin{gather*}
  \sigma_0 = \bigl( \bmat 1 & \\ & -\one \emat, \one \bigr) \quad \text{and} \quad
  \sigma_1 = \Bigl( \bmat 0 & -1 & \\ 1 & 0 & \\ & & \one \emat, \bmat 0 & -1 \\ 1 & 0 \emat \Bigr).
\end{gather*}
In order to obtain formulas for the two involutions $\tilde\sigma_0$ and $\tilde \sigma_1$ that generate the Weyl group of the action on $\OO(2m_0+4)$, we note that
\begin{gather*}
  \Theta(\sigma_0) = \bmat 1 & & & & \\ & -1 & & & \\ & & -1 & & \\ & & & 1 & \\ & & & & -\one \emat \quad \text{and} \quad
  \Theta(\sigma_1) = \bmat 0 & 1 & & & \\ 1 & 0 & & & \\ & & 0 & -1 & & \\ & & -1 & 0 & & \\ & & & & 0 & -\one \\ & & & & \one & 0 \emat.
\end{gather*}
For convenience, we denote $\Theta(\sigma_0)$ and $\Theta(\sigma_1)$ again by $\sigma_0$ and $\sigma_1$. The composition
\begin{gather*}
  \rho = \sigma_1\cdot \sigma_0 = \bmat 0 & -1 & & & & \\ 1 & 0 & & & & \\ & & 0 & -1 & & \\ & & 1 & 0 & & \\ & & & & 0 & -\one \\ & & & & \one & 0 \emat
\end{gather*}
is a primitive rotation of order $4$ in the dihedral Weyl group $D_4$ of the $\OO(m_0+2) \times \OO(2)$-action on $\Sph^{2m_0+3}$.
Now we have $\tilde \sigma_0 = (\sigma_0,\sigma_0)$ and $\tilde \sigma_1 = (\sigma_1,\hat\sigma_1)$ where
$\hat\sigma_1 = \tilde\gamma(-\tfrac{\pi}{4}) \sigma_1 \tilde\gamma(\tfrac{\pi}{4})$, and, hence,
\begin{gather*}
  \tilde \rho = (\rho,\hat \rho), \text{ where }
    \hat \rho = \hat \sigma_1 \cdot \sigma_0 = \bmat \one_2 & & & & \\ & 0 & -1 & & \\ & 1 & 0 & & \\ & & & 0 & -\one \\ & & & \one & 0 \emat
\end{gather*}
is a primitive rotation of order $4$ in the dihedral Weyl group $D_4$ of the $\OO(m_0+2) \times \OO(2) \times \OO(2m_0+3)$-action on $\OO(2m_0+4)$.
The following statement is now evident.

\begin{lemma}
The vector $\tilde F$ is invariant under the Weyl group rotation, i.e., $\hat \rho \tilde F \hat \rho^{-1} = \tilde F$.
\end{lemma}

The Lie algebra of $G = \OO(m_0+2) \times \OO(2)$ splits orthogonally as
\begin{gather*}
  \ag = \ah \oplus \am_0 \oplus \am_1 \oplus \am_2 \oplus \am_3
\end{gather*}
where
\begin{multline*}
  \ah = \aso(m_0) \times \{0\}, \quad
  \am_0 = \Bigl\{ \Bigl( \bmat 0 & 0 & 0 \\ 0 & 0 & -u\tp \\ 0 & u & 0 \emat,  0\Bigr) \;\vert\; u \in \R^{m_0}\Bigr\}, \\
  \am_1 = \R \cdot \Bigl( \bmat 0 & -1 & 0 \\ 1 & 0 & 0 \\ 0 & 0 & 0 \emat,  \bmat 0 & -1 \\ 1 & 0 \emat\Bigr), \quad
  \am_2 = \rho \am_0 \rho^{-1}, \quad
  \am_3 = \rho \am_1 \rho^{-1}.
\end{multline*}
The principal isotropy group $H$ acts by multiplication by $\epsilon_1\epsilon_2$ on $\am_1 \oplus \am_3$, on $\am_0$ by multiplication by $\epsilon_2 C$ on $\R^{m_0}$, and on $\am_2$ by multiplication by $\epsilon_1 C$ on $\R^{m_0}$.

\smallskip

Under the derivative of the homomorphism $\Theta$ at $(\one,\one)$ the generators
of $\am_0$ and $\am_1$ are mapped as follows:

\begin{gather*}
  \xi_0(u) := d\Theta_{\vert(\one,\one)} \Bigl( \bmat 0 & 0 & 0 \\ 0 & 0 & -u\tp \\ 0 & u & 0 \emat,  0\Bigr) =
    \bmat & & & & 0 & 0 \\ & & & & -u\tp & 0 \\ & & & & 0 & -u\tp \\ & & & & 0 & 0 \\ 0 & u & 0 & 0 & & \\ 0 & 0 & u & 0 & &\emat,\\
  \xi_1 := d\Theta_{\vert(\one,\one)} \Bigl( \bmat 0 & -1 & 0\\ 1 & 0 & 0 \\ 0 & 0 & 0\emat, \bmat 0 & -1 \\ 1 & 0 \emat \Bigr) =
  \bmat 0 & 0 & -1 & -1 & & \\ 0 & 0 & 1 & 1 & & \\ 1 & -1 & 0 & 0 & & \\ 1 & -1 & 0 & 0 & & \\ & & & & 0 & \one \\ & & & & -\one & 0 \emat.
\end{gather*}
We further need
\begin{gather*}
  \hat \xi_1 := \tilde\gamma(-\tfrac{\pi}{4}) \xi_1 \tilde\gamma(\tfrac{\pi}{4})
   = \bmat 0 & 0 & 0 & 0 & & \\ 0 & 0 & \sqrt{2} & \sqrt{2} & & \\ 0 & -\sqrt{2} & 0 & 0 & & \\ 0 & -\sqrt{2} & 0 & 0 & & \\ & & & & 0 & \one \\ & & & & -\one & 0 \emat.
\end{gather*}
Now, $\tilde\am_0$ is generated by the $(\xi_0(u),\xi_0(u))$ and $\tilde\am_1$ is generated by $(\xi_1,\hat\xi_1)$.
The other $\tilde\am_i$ are given by
\begin{gather*}
  \tilde\am_{2\ell} = \tilde\rho^{\ell} \tilde\am_0 \tilde\rho^{-\ell} \quad \text{and} \quad \tilde\am_{2\ell+1} = \tilde\rho^{\ell} \tilde\am_1 \tilde\rho^{-\ell}.
\end{gather*}

\smallskip

\begin{theorem}
The tangential component of the tension field of any $(k,r)$-map vanishes for the $\OO(m_0+2)\times \OO(2)\times \OO(2m_0+3)$-action on $\OO(2m_0+4)$.
\end{theorem}

\begin{proof}
Let $N = \dim \OO(2m_0+4)-1$. We take $N$ vectors $\tilde E_1, \ldots, \tilde E_N$ compatible with the sum
\begin{gather*}
  \tilde\am_0 \oplus \ldots \oplus \tilde \am_{7} \oplus \aso(2m_0+2)
\end{gather*}
such that the 
$\tilde E_1^{\ast}, \ldots, \tilde E_N^{\ast}$ are orthonormal at $\tilde \gamma(t)$ and such that $\tilde E_N = \tilde F$.
By Lemma\,\ref{leminvariant} and Theorem\,\ref{tanpart} we have
\begin{gather*}
  \tau^{\tan}_{\vert\gamma(t)} = -\sum_{\mu=1}^N \langle [\tilde E_{\mu},\tilde F]^{\ast},\tilde E_{\mu}^{\ast} \rangle_{\vert\gamma(r(t))} \tilde F^{\ast}_{\vert\gamma(r(t))}.
\end{gather*}
It is now straightforward to compute that
\begin{gather*}
  [\xi_0(u), \tilde F] = \bmat & & & & 0 & 0 \\ & & & & 0 & 0 \\ & & & & 0 & 0 \\ & & & & 0 & u\tp \\ 0 & 0 & 0 & 0 & & \\ 0 & 0 & 0 & -u & &\emat
  \in \aso(2m_0+2) \subset \aso(2m_0+4).
\end{gather*}
By the invariance of $\tilde F$ and $\aso(2m_0+2)$ under the Weyl group rotation we get
\begin{gather*}
  [\tilde \am_{2\ell}, \tilde F] \subset \aso(2m_0+2).
\end{gather*}
Another straightforward computation shows
\begin{gather*}
  [\hat\xi_1, \tilde F] = \tfrac{1}{2} (\hat\rho \hat\xi_1 \hat\rho^{-1} - \hat\rho^3\hat \xi_1 \hat\rho^{-3}).
\end{gather*}
Hence,
\begin{gather*}
  [\tilde \am_{2\ell+1}, \tilde F] \subset \tilde\am_{2\ell+3} \oplus \tilde\am_{2\ell+7}.
\end{gather*}
It follows that
\begin{gather*}
  \langle [\tilde E_{\mu},\tilde F]^{\ast},\tilde E_{\mu}^{\ast} \rangle_{\vert\gamma(r(t))} = 0
\end{gather*}
if $\tilde E_{\mu}$ is contained in some $\tilde \am_i$. The same is true for $\tilde E_{\mu}$ in $0 \times \aso(2m_0+2)$. Indeed, $[\tilde E_{\mu},\tilde F]$ is perpendicular to $\tilde E_{\mu}$ in $\aso(2m_0+2)$ since $\ad_{\tilde F}$ is skew-symmetric, and $\aso(2m_0+2) \to \aso(2m_0+2)^{\ast}_{\tilde\gamma(r(t))}$ is an isometry. All in all, we have shown that $\tau^{\tan} = 0$.
\end{proof}

\smallskip

\subsection{The case $(g,m) = (6,1)$}
\label{sixone}
Let $\H = \C \oplus j\C$ denote the algebra of quaternions, $\Syp(1)$ the group of unit quaternions, and $\Syp(2)$ the group of quaternionic $2\times 2$-matrices $A$ with $\bar A\tp A = \one$. We consider the homomorphism
\begin{gather*}
  \vartheta: \Syp(1) \to \Syp(2), \quad a + jb \mapsto \bmat \bar a^3 + j \bar b^3 & \sqrt{3} (\bar ab^2+ja^2\bar b)\\
    \sqrt{3} (\bar a\bar b^2 + j\bar a^2 \bar b) & \;a (\abs{a}^2 - 2\abs{b}^2) + j b(\abs{b}^2 - 2\abs{a}^2) \emat
\end{gather*}
and the action of $\Syp(1)\times \Syp(1)^3$ on $\Sph^7 \subset \H^2$
\begin{gather*}
  (q_1,q_2)\cdot u = \vartheta(q_1)u \bar q_2.
\end{gather*}
This is the up to equivalence unique $(6,1)$-action on $\Sph^7$. Note that the ineffective kernel of this action is $\{\pm(1,1)\}$.

\smallskip

A normal geodesic for this action is
\begin{gather*}
  \gamma(t) = \bmat \cos t\\ \sin t\emat.
\end{gather*}
 The principal isotropy along $\gamma$ is
\begin{gather*}
  H = \{\pm (1,1), \pm (i,i), \pm (j,j), \pm (k,k)\}.
\end{gather*}
Non-principal isotropy groups along $\gamma$ appear at the multiples of $\tfrac{\pi}{6}$: 
\begin{xalignat*}{2}
  t &= 0 \mod \pi: & K_0 &= \{ (e^{it},e^{-3it})\} \cup \{ j(e^{it},e^{-3it})\},\\
  t &= \tfrac{\pi}{6} \mod \pi: & K_1 &= \{ (e^{jt},e^{jt})\} \cup \{k(e^{jt},e^{jt})\},\\
  t &= \tfrac{\pi}{3} \mod \pi: & K_2 &= \{ (e^{kt},e^{-3kt})\} \cup \{ i(e^{kt},e^{-3kt})\},\\
  t &= \tfrac{\pi}{2} \mod \pi: & K_3 &= \{ (e^{it},e^{it})\} \cup \{ j(e^{it},e^{it})\},\\
  t &= \tfrac{2\pi}{3} \mod \pi: & K_4 &= \{ (e^{jt},e^{3jt})\} \cup \{k(e^{jt},e^{-3jt})\},\\
  t &= \tfrac{5\pi}{6} \mod \pi: & K_5 &= \{ (e^{kt},e^{kt})\} \cup \{ i(e^{kt},e^{kt})\}.
\end{xalignat*}
Via the above action $\Syp(1) \times \Syp(1)$ becomes a subgroup of $\SO(\H^2)$.
From now on we use the basis
\begin{gather*}
  \bmat 1\\0 \emat, \bmat 0 \\ 1\emat, \bmat i\\0 \emat, \bmat 0 \\ i\emat, \bmat j\\0 \emat, \bmat 0 \\ j\emat, \bmat -k\\0 \emat, \bmat 0 \\ -k\emat
\end{gather*}
of $\H^2$ to identify $\SO(\H^2)$ with $\SO(8)$. The homomorphism
\begin{gather*}
  \Theta: \Syp(1) \times \Syp(1) \to \SO(8)
\end{gather*}
maps the elements of the principal isotropy group $H$ as follows:
\begin{multline*}
  \pm (1,1) \mapsto \one_8, \quad
  \pm (i,i) \mapsto \bmat \one_4 & \\ & -\one_4 \emat,\\
  \pm (j,j) \mapsto \bmat \one_2 & & & \\ & -\one_2 & & \\ & & \one_2 & \\ & & & -\one_2 \emat,
  \pm (k,k) \mapsto \bmat  \one_2 & & & \\ & -\one_2 & & \\ & & -\one_2 & \\ & & & \one_2 \emat.
\end{multline*}
The following statement is now evident.

\begin{lemma}
The fixed point set $(\Sph^7)^H$ just consists of the unparametrized normal geodesic $\gamma(\R)$.
\end{lemma}

\begin{corollary}
The tangential part of the tension field of any $(k,r)$-map vanishes for the $\Syp(1)\times\Syp(1)$-action on $\Sph^7$ vanishes.
\end{corollary}

\smallskip

The $\Syp(1)\times\Syp(1)$-action on $\Sph^7$ lifts to the $\Syp(1)\times\Syp(1)\times \SO(7)$-action on $\SO(8)$ given by
\begin{gather*}
  (q_1,q_2,B) \cdot C = \Theta(q_1,q_2) C B^{-1}.
\end{gather*}
Note that we use the metric $\langle X,Y\rangle = \tfrac{1}{2} \tr X\tp Y$ on $\SO(8)$. A normal geodesic is
\begin{gather*}
  \tilde \gamma(t) = \bmat \cos t & -\sin t & 0\\ \sin t & \cos t & 0\\ 0 & 0 & \one_6\emat
\end{gather*}
with principal isotropy group $\tilde H = \Delta H$.

\begin{lemma}
The fixed point set of $\tilde H$ in the tangent space $T_{\tilde\gamma(t)} \tilde G \cdot \tilde \gamma(t)$ is generated by the three vectors $\tilde F^{\ast}_{1\vert\tilde\gamma(t)}$, $\tilde F^{\ast}_{2\vert\tilde\gamma(t)}$, $\tilde F^{\ast}_{3\vert\tilde\gamma(t)}$ where
\begin{gather*}
  \tilde F_1 := \bmat 0 & & & \\ & J & & \\ & & 0 & \\ & & & 0\emat, \quad
  \tilde F_2 := \bmat 0 & & & \\ & 0 & & \\ & & J & \\ & & & 0\emat, \quad
  \tilde F_3 := \bmat 0 & & & \\ & 0 & & \\ & & 0 & \\ & & & J\emat \in \aso(6) \subset \aso(8)
\end{gather*}
with $J = \bmat 0 & -1\\ 1 & 0\emat$.
\end{lemma}

Note that $\tilde F_1^{\ast}$, $\tilde F_2^{\ast}$, and $\tilde F_3^{\ast}$ are orthonormal parallel vector fields along $\tilde \gamma$.

\begin{proof}
The fixed point set of $\tilde H$ in $\SO(8)$ can easily be seen to consist of elements of the form
\begin{gather*}
  \bmat \ast_2 & & & \\ & \ast_2 & & \\ & & \ast_2 & \\ & & & \ast_2 \emat.
\end{gather*}
The tangent space to this fixed point set at $\tilde\gamma(t)$ is spanned by
$\dot{\tilde\gamma}(t) = \bmat J & & & \\ & 0 & & \\ & & 0 & \\ & & & 0\emat$, $\tilde F^{\ast}_{1\vert\tilde\gamma(t)}$,
$\tilde F^{\ast}_{2\vert\tilde\gamma(t)}$, and $\tilde F^{\ast}_{3\vert\tilde\gamma(t)}$.
\end{proof}

\smallskip

The Weyl group of the $\Syp(1)\times\Syp(1)$-action on $\Sph^7$ is generated by
\begin{gather*}
  \sigma_0 = (e^{i\pi/4},e^{-i3\pi/4}) \quad\text{and}\quad \sigma_1 = (e^{j\pi/4},e^{j\pi/4}).
\end{gather*}
In order to obtain formulas for the two involutions $\tilde\sigma_0$ and $\tilde \sigma_1$ that generate the Weyl group of the action on $\SO(8)$, we note that
\begin{gather*}
  \Theta(\sigma_0) = \bmat S & 0 & 0 & 0\\ 0 & S & 0 & 0\\ 0 & 0 & 0 & S \\ 0 & 0 & -S & 0 \emat \quad \text{and} \quad
  \Theta(\sigma_1) = \bmat RS & 0 & 0 & 0 \\ 0 & 0 & 0 & -RS \\ 0 & 0 & RS & 0 \\ 0 & RS & 0 & 0 \emat
\end{gather*}
where $S = \bmat 1 & 0\\ 0 & -1 \emat$ and $R$ is the matrix for a counter-clockwise rotation by an angle of $\pi/3$. For convenience, we denote $\Theta(\sigma_0)$ and $\Theta(\sigma_1)$ again by $\sigma_0$ and $\sigma_1$. The composition
\begin{gather*}
  \rho = \sigma_1\cdot \sigma_0 = \bmat R & 0 & 0 & 0\\ 0 & 0 & R & 0 \\ 0 & 0 & 0 & R\\ 0 & R & 0 & 0\emat
\end{gather*}
is a primitive rotation of order $6$ in the dihedral Weyl group $D_6$ of the $\Syp(1)\times \Syp(1)$-action on $\Sph^7$.
Now we have $\tilde \sigma_0 = (\sigma_0,\sigma_0)$ and $\tilde \sigma_1 = (\sigma_1,\hat\sigma_1)$ where
$\hat\sigma_1 = \tilde\gamma(-\tfrac{\pi}{6}) \sigma_1 \tilde\gamma(\tfrac{\pi}{6})$, and, hence,
\begin{gather*}
  \tilde \rho = (\rho,\hat \rho), \text{ where }
    \hat \rho = \hat \sigma_1 \cdot \sigma_0 = \bmat \one & 0 & 0 & 0 \\ 0 & 0 & R & 0\\ 0 & 0 & 0 & R\\ 0 & R & 0 & 0\emat
\end{gather*}
is a primitive rotation of order $6$ in the dihedral Weyl group $D_6$ of the $\Syp(1)\times \Syp(1)\times \SO(7)$-action on $\SO(8)$. The following statement can be easily verified.

\begin{lemma}
The abelian Lie subalgebra $\tilde \af$ of $\aso(6)$ generated by $\tilde F_1$, $\tilde F_2$ and $\tilde F_2$ is invariant under the Weyl group rotation $\tilde \rho$. More precisely,
\begin{gather*}
  \hat \rho \tilde F_1 \hat \rho^{-1} = \tilde F_3, \quad \hat \rho \tilde F_2 \hat \rho^{-1} = \tilde F_1, \quad \hat \rho \tilde F_3 \hat \rho^{-1} = \tilde F_2.
\end{gather*}
\end{lemma}

Since $H$ is finite, the Lie algebra of $G = \Syp(1) \times \Syp(1)$ splits orthogonally as
\begin{gather*}
  \ag = \am_0 \oplus \am_1 \oplus \am_2 \oplus \am_3 \oplus \am_4 \oplus \am_5
\end{gather*}
and we have $\am_i = \ak_i$ for all $i$. Hence,
\begin{gather*}
  \am_0 = \R\cdot (i,-3i), \quad
  \am_1 = \R\cdot (j,j), \quad
  \am_2 = \R\cdot (k,-3k), \\
  \am_3 = \R\cdot (i,i), \quad
  \am_4 = \R\cdot (j,-3j), \quad
  \am_5 = \R\cdot (k,k).
\end{gather*}
The principal isotropy group $H$ acts by conjugation on $\am$. The elements $\pm (i,i)$ act trivially on $\am_0\oplus \am_3$ and by $-\id$ on the other $\am_i$,
$\pm (j,j)$ act trivially on $\am_1\oplus\am_4$ and by $-\id$ on the other $\am_i$, and $\pm (k,k)$ act trivially on $\am_2\oplus \am_5$ and by $-\id$ on the other $\am_i$.

\smallskip

Under the derivative of the homomorphism $\Theta: \Syp(1) \times \Syp(1) \to \SO(8)$ at $(1,1)$ the generators $(i,-3i)$ of $\am_0$ and $(j,j)$ of $\am_1$ are mapped to
\begin{gather*}
  \xi_0 := d\Theta_{\vert(1,1)}(i,-3i) = \bmat & &  0 & 0 &  & &  & \\ & &  0 & -4 &  & &  & \\ 0 & 0 &  & &  & &  & \\ 0 & 4 &  & &  & &  & \\
  & &  & &  & & -6 & 0\\ & &  & &  & &  0 & -2 \\ & &  & &  6 & 0 &  & \\ & &  & &  0 & 2 &  & \emat
\end{gather*}
and
\begin{gather*}
  \xi_1 := d\Theta_{\vert(1,1)}(j,j) = \bmat & &  & &  1 & -\sqrt{3} &  & \\ & &  & &  -\sqrt{3} & 3 &  & \\ & &  & &  & &  -1 & -\sqrt{3}\\ & &  & &  & &  -\sqrt{3} & 1\\
  -1 & \sqrt{3} &  & &  & &  & \\ \sqrt{3} & -3 &  & &  & &  & \\ & &  1 & \sqrt{3} &  & &  & \\ & &  \sqrt{3} & -1 &  & &  & \emat.
\end{gather*}
We further need
\begin{gather*}
  \hat \xi_1 := \tilde\gamma(-\tfrac{\pi}{6}) \xi_1 \tilde\gamma(\tfrac{\pi}{6})
   = \bmat & &  & &  0 & 0 &  & \\ & &  & &  -2 & 2\sqrt{3} &  & \\ & &  & &  & &  -1 & -\sqrt{3}\\ & &  & &  & &  -\sqrt{3} & 1\\
  0 & 2 &  & &  & &  & \\ 0 & -2\sqrt{3} &  & &  & &  & \\ & &  1 & \sqrt{3} &  & &  & \\ & &  \sqrt{3} & -1 &  & &  & \emat.
\end{gather*}
Now, the Lie algebra $\tilde\ak_0 = \tilde\am_0$ of $\tilde K_0$ is generated by $(\xi_0,\xi_0)$ and the Lie algebra $\tilde\ak_1 = \tilde\am_1$ of $\tilde K_1$ is generated by $(\xi_1,\hat\xi_1)$. The Lie algebras $\tilde\ak_i = \tilde\am_i$ of the other singular isotropy groups are given by
\begin{gather*}
  \tilde\am_{2\ell} = \tilde\rho^{\ell} \tilde\am_0 \tilde\rho^{-\ell} \quad \text{and} \quad \tilde\am_{2\ell+1} = \tilde\rho^{\ell} \tilde\am_1 \tilde\rho^{-\ell}.
\end{gather*}

\begin{theorem}
The tangential component of the tension field of any $(k,r)$-map vanishes for the $\Syp(1)\times\Syp(1)\times \SO(7)$-action on $\SO(8)$.
\end{theorem}

\begin{proof}
We first note that $\dim \SO(8)-1 = 27$. We take $27$ vectors $\tilde E_1, \ldots, \tilde E_{27}$ compatible with the sum
\begin{gather*}
  \tilde\am_0 \oplus \ldots \oplus \tilde \am_{11} \oplus \aso(6)
\end{gather*}
such that
$\tilde E_1^{\ast}, \ldots, \tilde E_{27}^{\ast}$ are orthonormal at $\tilde \gamma(t)$ and such that $\tilde E_{25} = \tilde F_1$, $\tilde E_{26} = \tilde F_2$, $\tilde E_{27} = \tilde F_3$.
By Theorem\,\ref{tanpart} and Lemma\,\ref{leminvariant} we have
\begin{gather*}
  \tau^{\tan}_{\vert\gamma(t)} = -\sum_{\nu = 1}^3\sum_{\mu=1}^{27} \langle [\tilde E_{\mu},\tilde F_{\nu}]^{\ast},\tilde E_{\mu}^{\ast} \rangle_{\vert\gamma(r(t))} \tilde F^{\ast}_{\nu\vert\gamma(r(t))}.
\end{gather*}
It is now straightforward to compute that
\begin{gather*}
  [\xi_0, \tilde F_1] = \tfrac{1}{2} (\hat\rho \xi_1 \hat\rho^{-1} - \hat\rho^4\xi_1 \hat\rho^{-4})
\end{gather*}
and that $[\xi_0, \tilde F_2]$ and $[\xi_0, \tilde F_3]$ are contained in $\aso(6)$. Hence,
$[\tilde \am_0, \tilde \af] \subset \tilde\am_3 \oplus \tilde\am_9 \oplus\aso(6)$.
Since $\tilde\af$ and $\tilde\aso(6)$ are both invariant under conjugation by $\tilde\rho$ this implies
\begin{gather*}
  [\tilde \am_{2\ell}, \tilde \af] \subset \tilde\am_{2\ell+3} \oplus \tilde\am_{2\ell+9} \oplus\aso(6).
\end{gather*}
Similarly,
\begin{gather*}
  [\hat\xi_1, \tilde F_2] = \tfrac{1}{2} (\hat\rho^2 \xi_0 \hat\rho^{-2} - \hat\rho^5\xi_0 \hat\rho^{-5})
\end{gather*}
and $[\hat\xi, \tilde F_1]$ and $[\hat\xi_1, \tilde F_3]$ are contained in $\aso(6)$. Hence,
\begin{gather*}
  [\tilde \am_{2\ell+1}, \tilde \af] \subset \tilde\am_{2\ell+4} \oplus \tilde\am_{2\ell+10} \oplus\aso(6).
\end{gather*}
It follows that
\begin{gather*}
  \langle [\tilde E_{\mu},\tilde F_{\nu}]^{\ast},\tilde E_{\mu}^{\ast} \rangle_{\vert\gamma(r(t))} = 0
\end{gather*}
if $\tilde E_{\mu}$ is contained in some $\tilde \am_i$. The same is true for $\tilde E_{\mu}$ in $0 \times \aso(6)$. Indeed, $[\tilde E_{\mu},\tilde F_{\nu}]$ is perpendicular to $\tilde E_{\mu}$ in $\aso(6)$ since $\ad_{\tilde F_{\nu}}$ is skew-symmetric, and $\aso(6) \to \aso(6)^{\ast}_{\tilde\gamma(r(t))}$ is an isometry. All in all, we have shown that $\tau^{\tan} = 0$.
\end{proof}

\bigskip

\section{A trigonometric identity}
\label{trigidentity}
The following identity and its derivative are used to evaluate the tension fields of the $(k,r)$-maps for the $(g,m)$- and $(g,m_0,m_1)$-actions on $\Sph^{n+1}$ and $\SO(n+2)$.

\begin{lemma}
\label{2para}
For every non-zero integer $g$ and all $r,t\in\R$ we have
\begin{gather*}
  \sum_{i=0}^{g-1}\frac{\sin^2(r-i\tfrac{\pi}{g})}{\sin^2(t-i\tfrac{\pi}{g})} \sin^2 gt
  = g \bigl( (g-1)\sin^2(r-t) + \sin^2(r+(g-1)t)\bigr).
\end{gather*}
\end{lemma}

\begin{proof}
Substituting $\rho = r-t$ the identity above transforms into the equivalent identity
\begin{gather}
\label{transformed}
  \sum_{i=0}^{g-1}\frac{\sin^2(\rho+t-i\tfrac{\pi}{g})}{\sin^2(t-i\tfrac{\pi}{g})} \sin^2 gt
  = g \bigl( (g-1)\sin^2\rho + \sin^2(\rho+gt)\bigr).
\end{gather}
Notice that both sides of this equation are functions of the form $c + A\cos 2\rho + B\sin 2\rho$.
Hence, the goal is to show that the coefficients $c$, $A$, and $B$ are the same on both sides.

We first transform the right-hand side of the equation. Expanding $\sin^2 \rho$ to $\frac{1}{2}(1-\cos 2\rho)$
and $\sin^2 (\rho+gt)$ analogously  and applying the addition formulas for the cosine function yields
\begin{gather*}
  g \bigl( (g-1)\sin^2\rho + \sin^2(\rho+gt)\bigr)
  = \frac{g^2}{2} - \frac{g}{2}(g-2\sin^2 gt) \cos 2\rho - \frac{g}{2} (\sin 2gt) \sin 2\rho.
\end{gather*}
On the left-hand side we similarly obtain
\begin{multline*}
  \sum_{i=0}^{g-1}\frac{\sin^2(\rho+t-i\tfrac{\pi}{g})}{\sin^2(t-i\tfrac{\pi}{g})} \sin^2 gt\\
  = \frac{1}{2} \sum_{i=0}^{g-1}\frac{\sin^2 gt}{\sin^2(t-i\tfrac{\pi}{g})}
   - \frac{1}{2} \sum_{i=0}^{g-1}\frac{\sin^2 gt}{\sin^2(t-i\tfrac{\pi}{g})} \,\cos 2(t-i\tfrac{\pi}{g})\,\cos 2\rho\\
   - \frac{1}{2} \sum_{i=0}^{g-1}\frac{\sin^2 gt}{\sin^2(t-i\tfrac{\pi}{g})} \,\sin 2(t-i\tfrac{\pi}{g})\, \sin 2\rho.
\end{multline*}

In order to see that the coefficients are the same on both sides, we use the standard cotangent identity
\begin{gather}
\label{cot}
  g\cot gt = \sum_{i=0}^{g-1} \cot(t-i\tfrac{\pi}{g})
\end{gather}
Differentiating this identity on both sides yields
\begin{gather*}
  g^2 = \sum_{i=0}^{g-1}\frac{\sin^2 gt}{\sin^2(t-i\tfrac{\pi}{g})}
\end{gather*}
which shows that the constant terms on both sides of (\ref{transformed}) are the same.
Now, the coefficients of $\cos 2\rho$ in (\ref{transformed}) are the same since
\begin{multline*}
  g (g-2\sin^2 gt) = \sum_{i=0}^{g-1}\frac{\sin^2 gt}{\sin^2(t-i\tfrac{\pi}{g})}
    - 2\sum_{i=0}^{g-1}\frac{\sin^2 gt}{\sin^2(t-i\tfrac{\pi}{g})}\sin^2(t-i\tfrac{\pi}{g})\\
    = \sum_{i=0}^{g-1}\frac{\sin^2 gt}{\sin^2(t-i\tfrac{\pi}{g})} \,\cos 2(t-i\tfrac{\pi}{g}).
\end{multline*}
Finally, multiplication of the standard cotangent identity by $2\sin gt$ shows that the coefficients of $\sin 2\rho$
in (\ref{transformed}) are the same.
\end{proof}

\begin{lemma}
\label{2parac}
For every nonzero integer $g$ and all $r,t\in\R$ we have
 \begin{align*}
  \sum_{i=0}^{g-1}\frac{\sin 2(r-i\frac{\pi}{g})}{\sin^2(t-i\tfrac{\pi}{g})}\sin^2(gt) =
  g \big((g-1)\sin 2(r-t) +\sin 2(r+(g-1)t)\big).
 \end{align*}
\end{lemma}
\begin{proof}
Differentiation of the identity of Lemma\,\ref{2para} with respect to $r$.
\end{proof}

\end{document}